\documentclass[3p]{elsarticle}

\usepackage{amsmath,amssymb,amsthm,mathtools}
\usepackage{graphicx}
\usepackage{float}
\usepackage[colorlinks=true,linkcolor=blue,citecolor=blue,urlcolor=blue]{hyperref}
\usepackage{tikz}
\usetikzlibrary{arrows.meta,calc,positioning,decorations.pathreplacing}
\usepackage{pgfplots}
\pgfplotsset{compat=1.18}

\newtheorem{theorem}{Theorem}[section]
\newtheorem{proposition}[theorem]{Proposition}
\newtheorem{lemma}[theorem]{Lemma}
\newtheorem{corollary}[theorem]{Corollary}
\theoremstyle{definition}
\newtheorem{definition}[theorem]{Definition}
\newtheorem{remark}[theorem]{Remark}

\newtheorem{question}[theorem]{Question}

\newcommand{\RR}{\mathbb{R}}
\newcommand{\diam}{\operatorname{diam}}
\newcommand{\len}{\operatorname{len}}
\newcommand{\im}{\operatorname{im}}
\newcommand{\Hg}{\mathcal{H}}
\newcommand{\Poly}{\mathcal{P}}
\newcommand{\eps}{\varepsilon}
\newcommand{\Rg}{R_g}
\newcommand{\sD}{\mathsf{D}}
\newcommand{\Thi}{\operatorname{Thi}}
\newcommand{\Rop}{\operatorname{Rop}}

\journal{Topology and its Applications}

\begin{document}
\begin{frontmatter}

\title{Ropelength-Filtered Knot Densities}

\author[komazawa]{Makoto Ozawa}
\ead{w3c@komazawa-u.ac.jp}
\affiliation[komazawa]{organization={Department of Natural Sciences, Faculty of Arts and Sciences, Komazawa University},
            city={Tokyo},
            country={Japan}}

\begin{abstract}
We study scale-invariant $p$-densities of knot types in Euclidean three-space, defined as the infimum over representatives of the length divided by an $L^p$-type mean of the pairwise chord lengths; the exponent $p$ is any value greater than minus one, including infinity, where the mean is the diameter.  The unconstrained density is independent of the knot type.  For exponents at most two, its common value is the density of the round circle, by the sharp mean-chord inequality; at the diameter endpoint it equals two; and for finite exponents greater than two it reduces to a planar convex extremal problem, in which the round circle fails to be extremal beyond an explicit threshold.

To prevent this degeneration, which is caused by knotting at arbitrarily small scale, we introduce a ropelength filtration: at each level $\lambda$, the infimum is restricted to representatives whose ropelength is at most $\lambda$ times the ropelength of the knot type.  The filtered density has minimizers, detects the unknot for exponents at most two, and converges to the unconstrained density as $\lambda$ tends to infinity.  It is continuous in the exponent, including the diameter endpoint, and right-continuous in the filtration parameter.

As a numerical test, we evaluate the density on published approximately tight configurations of the granny knot and the square knot.  The sampled values differ by about 2.5 percent at exponent two, whereas their ordering is reversed at the diameter endpoint.  These computations provide evidence that the ropelength filtration retains geometric information lost by the unconstrained theory.
\end{abstract}


\end{frontmatter}

\section{Introduction}

We study scale-invariant quantities that compare the length of a knot representative with the spatial spread of the curve.  The chord-moment functionals underlying our construction are classical rather than new.  L\"uk\H{o} studied mean-chord functionals for closed curves \cite{Luko}; in the knot-energy literature, Abrams--Cantarella--Fu--Ghomi--Howard treated double-integral chord functionals of the form $\iint f(|\gamma(s)-\gamma(t)|^2)\,ds\,dt$ and explicitly considered $\iint |\gamma(s)-\gamma(t)|^p\,ds\,dt$ \cite[eqs.~(1.4)--(1.5)]{ACFGH}; and Exner--Harrell--Loss established sharp $p$-mean chord inequalities in the range $p\le2$ \cite{EHL}.  We package these chord moments as a normalized generalized mean of all pairwise chord lengths and use it as the spread variable in a knot-density construction.

If $\gamma\subset \RR^3$ is a tame $C^1$ embedded closed curve and $L=\len(\gamma)$, we write, for $p\in(-1,\infty)\setminus\{0\}$,
\[
\sD_p(\gamma):=
\left(
\frac{1}{L^2}\int_\gamma\int_\gamma |x-y|^p\,ds_x\,ds_y
\right)^{1/p}.
\]
For $p=0$ we use the corresponding geometric mean, the continuous extension of the $p$-means as $p\to0$ (Proposition~\ref{prop:pzero-limit}), and for $p=\infty$ we use the limiting generalized mean $\sD_\infty(\gamma)=\diam(\gamma)$.  Thus the notation $\sD_p$ is a convenient normalization and unification of established mean-chord quantities; the new objects studied here are the knot-type density obtained from $\len/\sD_p$ and, in particular, its ropelength-filtered refinement.  The associated unconstrained density of a knot type $K$ is
\[
\rho_p(K):=
\inf_{\gamma\in\Hg(K)}\frac{\len(\gamma)}{\sD_p(\gamma)},
\]
where $\Hg(K)$ denotes the set of tame $C^1$ representatives of $K$.  Throughout the paper, knot types are ambient isotopy classes of tame embeddings of the circle in $\RR^3$, and $U$ denotes the unknot.

At first sight, this construction appears to provide a one-parameter family of geometric knot invariants.  The case $p=2$ is related to the radius of gyration, a standard measure of spatial spread in polymer physics, while the endpoint $p=\infty$ is the diameter density.  The negative range $-1<p<0$ is a mildly singular distance-average regime, adjacent to the more strongly singular inverse-power kernels appearing in knot-energy theory.

The first main result is that the unconstrained invariant does not detect knotting at all.  A knot may be tied inside an arbitrarily small ball on a nearly optimal curve for the unknot, and this changes both the length and $\sD_p$ by an arbitrarily small amount.  To state the result, let
\[
c_p:=\pi\left(
\frac{\pi}{\int_0^\pi \sin^p\theta\,d\theta}
\right)^{1/p}\quad(p\ne0),
\qquad
c_0:=2\pi,
\]
be the density of the round circle, and let
\[
A_p^{\mathrm{all}}:=\sup_\gamma
\frac{\sD_p(\gamma)}{\len(\gamma)},
\]
where the supremum is over all tame $C^1$ embedded closed curves, irrespective of knot type.

\begin{theorem}[Degeneration of the unconstrained densities]\label{thm:intro-degeneration}
Let $K$ and $J$ be knot types.
\begin{enumerate}
\item[(1)] For every $p\in(-1,\infty]$, one has $\rho_p(K\mathbin{\#}J)\le\rho_p(K)$; in particular $\rho_p(K)\le\rho_p(U)$.
\item[(2)] For $-1<p\le2$, one has $\rho_p(K)=c_p$; in particular $\rho_2(K)=\sqrt2\pi$.  The infimum is attained if and only if $K=U$, and then exactly by the round circles.
\item[(3)] For $2<p<\infty$, one has $\rho_p(K)=1/A_p^{\mathrm{all}}$.
\item[(4)] $\rho_\infty(K)=2$, and the infimum is not attained.
\end{enumerate}
Consequently $\rho_p(K)=\rho_p(U)$ for every knot type $K$ and every $p\in(-1,\infty]$.
\end{theorem}

Part~(1) is Theorem~\ref{thm:local-knotting}; parts~(2)--(4) are Theorem~\ref{thm:mean-chord-degeneration} with Corollary~\ref{cor:mean-chord-attainment}, Theorem~\ref{thm:finite-high-degeneration}, and Theorem~\ref{thm:infty-degeneration}, respectively.  In each range the lower bound for $\rho_p$ comes from an extremal inequality for closed curves that is blind to the knot type, and the matching upper bound comes from local knotting; see Figure~\ref{fig:exponent-ranges}.  For $-1<p\le2$ the inequality is the sharp mean-chord inequality of L\"uk\H{o} and Exner--Harrell--Loss \cite{Luko,EHL}, whose extremal curve is the round circle.  For $2<p<\infty$ it is Sallee's chord-stretching theorem \cite{Sallee}, which replaces every closed space curve by a convex plane curve of the same length while weakly increasing every corresponding chord; it was used by Abrams--Cantarella--Fu--Ghomi--Howard \cite{ACFGH} to show that the integrated $L^p$-mean chord functional has a convex planar maximizer.  In this range the common value is not explicit: its determination is a planar convex extremal problem, and we show that the round circle is not extremal for $p>p_*=3.5720\ldots$ (Proposition~\ref{prop:doubled-segment-threshold}).  At $p=\infty$ the inequality is $\len(\gamma)\ge2\diam(\gamma)$, approached by thin stadium curves.

\begin{figure}[t]
\centering
\begin{tikzpicture}[x=1cm,y=1cm,>=Stealth,font=\small]
\draw[gray,densely dashed] (-1.3,0)--(0,0);
\draw[thick] (0,0)--(10.2,0);
\draw[thick,dotted] (10.2,0)--(11.2,0);
\draw[thick,fill=white] (0,0) circle (2pt);
\fill (11.5,0) circle (2pt);
\foreach \x/\lab in {1.5/0,4.5/2}{\draw[thick] (\x,0.1)--(\x,-0.1) node[below]{$\lab$};}
\draw[thick] (7,0.1)--(7,-0.1) node[below]{$p_*$};
\node[below] at (0,-0.1) {$-1$};
\node[below] at (11.5,-0.1) {$\infty$};
\node[left,gray] at (-1.3,0) {$p$};
\draw[decorate,decoration={brace,amplitude=5pt}] (0.08,0.3)--(4.5,0.3);
\draw[decorate,decoration={brace,amplitude=5pt}] (4.6,0.3)--(11.1,0.3);
\node[align=center,anchor=south] at (2.29,0.55) {mean-chord inequality\\[1pt] $\rho_p(K)=c_p$};
\node[align=center,anchor=south] at (7.85,0.55) {Sallee convexification\\[1pt] $\rho_p(K)=1/A_p^{\mathrm{all}}$};
\node[align=center,anchor=south] at (11.5,0.55) {$\len\ge2\diam$\\[1pt] $\rho_\infty(K)=2$};
\draw[thick] (2.29,-1.45) circle (0.5);
\node[align=center,anchor=north] at (2.29,-2.05) {round circle\\ (attained iff $K=U$)};
\draw[thick] (7.85,-1.45) ellipse (0.95 and 0.38);
\node[align=center,anchor=north] at (7.85,-2.05) {convex plane curve\\ (not the circle for $p>p_*$)};
\draw[thick,rounded corners=1.6pt] (10.55,-1.51) rectangle (12.45,-1.39);
\node[align=center,anchor=north] at (11.5,-2.05) {doubled segment\\ (not attained)};
\end{tikzpicture}
\caption{The three regimes of Theorem~\ref{thm:intro-degeneration}: the knot-type-blind inequality giving the lower bound, the common value of $\rho_p(K)$, and the extremal shape.  The kernel $|x-y|^p$ is not integrable near the diagonal for $p\le-1$.}
\label{fig:exponent-ranges}
\end{figure}

Although the values collapse, attainment retains information: the infimum is attained exactly in the unknot class for $-1<p\le2$, in no knot class at $p=\infty$, and for $2<p<\infty$ any representative attaining the common value must be an unknot (Remark~\ref{rem:high-p-attainment}).  The unconstrained theory should therefore be read not as a source of knot invariants but as the identification of a degeneration mechanism: to obtain a non-degenerate theory, one must prevent arbitrarily small local knotting.

We therefore introduce a ropelength filtration.  Let $\Thi(\gamma)$ denote thickness, let $\Rop(\gamma):=\len(\gamma)/\Thi(\gamma)$, and let $\Rop(K)$ denote the ropelength of a knot type $K$.  For $\lambda\ge1$, define the \emph{ropelength-filtered $p$-density at level $\lambda$} by
\[
\rho^{\operatorname{rop}}_{p,\lambda}(K):=
\inf\left\{
\frac{\len(\gamma)}{\sD_p(\gamma)}\ \middle|\
\gamma\in\mathcal C^{1,1}(K),\
\Rop(\gamma)\le \lambda\Rop(K)
\right\},
\]
where $\mathcal C^{1,1}(K)$ denotes the class of embedded $C^{1,1}$ representatives of $K$.  After normalization to thickness one, the admissible class, which we call the \emph{ropelength window} of level $\lambda$, is cut out by a uniform thickness constraint together with a length cap.  These prevent both concentration of knotting at arbitrarily small scale and escape by excessive length.

\begin{theorem}[Ropelength-filtered densities]\label{thm:intro-filtered}
Let $K$ be a knot type, $p\in(-1,\infty]$, and $\lambda\ge1$.
\begin{enumerate}
\item[(1)] The infimum defining $\rho^{\operatorname{rop}}_{p,\lambda}(K)$ is attained.
\item[(2)] If $p\le2$, then $\rho^{\operatorname{rop}}_{p,\lambda}(K)\ge c_p$, with equality if and only if $K=U$.
\item[(3)] $\lim_{\lambda\to\infty}\rho^{\operatorname{rop}}_{p,\lambda}(K)=\rho_p(K)$.
\item[(4)] The function $p\mapsto\rho^{\operatorname{rop}}_{p,\lambda}(K)$ is continuous on $(-1,\infty]$, and $\lambda\mapsto\rho^{\operatorname{rop}}_{p,\lambda}(K)$ is non-increasing and right-continuous on $[1,\infty)$.
\end{enumerate}
\end{theorem}

These statements are Theorems~\ref{thm:rop-window-existence}, \ref{thm:rop-window-detects-unknot}, \ref{thm:lambda-limit}, and \ref{thm:window-p-continuity} together with Proposition~\ref{prop:lambda-right-continuity}.  Thus the filtered theory is compact, separates the unknot from every nontrivial knot type in the range where the sharp circle inequality is available, and exhausts the unconstrained theory.

We also give a numerical test of whether the filtered theory can see chirality data in a composite knot.  Using the published Ridgerunner configurations of Cantarella--LaPointe--Rawdon for the granny knot $3_1\#3_1$ and the square knot $3_1\#\overline{3_1}$ \cite{CLR,TightKnotData}, we compute their sampled tight-shape densities.  The values differ by about $2.47\%$ at $p=2$, while their ordering reverses at the diameter endpoint.  Since the input curves are polygonal numerical approximations to ropelength minimizers, this is evidence rather than a proof that the exact quantities $\rho^{\operatorname{rop}}_{p,1}$ distinguish the two knot types; see Subsection~\ref{subsec:granny-square}.

Finally, the unconstrained density is the limit of its polygonal counterparts: if $\rho_{p,n}(K)$ denotes the infimum over embedded polygonal $n$-gons representing $K$, then $\rho_{p,n}(K)\to\rho_p(K)$ as $n\to\infty$ (Theorem~\ref{thm:polygonal-approximation}).  At every fixed edge number, however, compactness gives a universal strict gap above the circle constant: for $p\in(0,2]$ there is a number $m_{p,n}>c_p$, independent of the knot type, such that $\rho_{p,n}(K)\ge m_{p,n}$ (Proposition~\ref{prop:universal-polygon-gap}).

The paper is organized as follows.  Section~\ref{sec:def} defines the spread functional and the unconstrained densities.  Section~\ref{sec:degeneration} proves Theorem~\ref{thm:intro-degeneration}.  Section~\ref{sec:rop-window} introduces the ropelength filtration, proves Theorem~\ref{thm:intro-filtered}, and presents the granny--square computation.  Section~\ref{sec:polygonal-approximation} treats polygonal approximation, and Section~\ref{sec:questions} records open problems.  The appendix collects the technical continuity and approximation lemmas.

\section{The spread functional and unconstrained densities}\label{sec:def}

\subsection{Admissible curves}

\begin{definition}
A \emph{continuous representative} of a knot type $K$ means a tame $C^1$ embedding
\[
\gamma\colon S^1\to \RR^3
\]
whose image represents $K$.
We denote by $\Hg(K)$ the set of all such representatives.
\end{definition}

\begin{remark}
Every $C^1$ embedded circle in $\RR^3$ is tame.  The adjective is retained only to emphasize that the knot types are the classical tame ambient-isotopy classes.
\end{remark}

\begin{remark}
Polygonal knots are not included in $\Hg(K)$, because they have corners.  They are used separately for the discrete densities.  This causes no loss in the comparison with the continuous theory: an embedded polygonal knot can be rounded inside arbitrarily small disjoint neighborhoods of its vertices, preserving the knot type and changing neither length nor the spread functional in the limit; see Lemma~\ref{lem:polygonal-smoothing-ratio}.
\end{remark}

\subsection{The spread functional}

The following notation packages the normalized pairwise chord moments recalled in the introduction.  The underlying double-integral functionals go back to the mean-chord literature \cite{Luko,ACFGH,EHL}; our purpose here is to place the finite moments, their geometric-mean limit, and the diameter endpoint in one scale-covariant family suited to the subsequent knot-type minimization and ropelength filtration.

\begin{definition}\label{def:spread}
Let $\gamma\in\Hg(K)$ and write $L=\len(\gamma)$.  For $p\in(-1,\infty)\setminus\{0\}$, set
\[
\sD_p(\gamma):=
\left(
\frac{1}{L^2}
\int_\gamma\int_\gamma |x-y|^p\,ds_x\,ds_y
\right)^{1/p}.
\]
For $p=0$, define
\[
\sD_0(\gamma):=
\exp\left(
\frac{1}{L^2}
\int_\gamma\int_\gamma \log|x-y|\,ds_x\,ds_y
\right).
\]
For $p=\infty$, define
\[
\sD_\infty(\gamma):=\diam(\gamma).
\]
\end{definition}

\begin{remark}
The restriction $p>-1$ ensures that the pairwise-distance integral is locally integrable near the diagonal for the $C^1$ and polygonal representatives considered here.
\end{remark}

\begin{remark}\label{rem:rectifiable-extension}
When no knot type is involved, the same formulas will also be used for a nonconstant rectifiable closed curve, understood as a Lipschitz arc-length parametrization of a parameter circle.  Positive moments and the diameter are always finite.  Negative and logarithmic moments are interpreted in the extended sense if an additional self-contact singularity makes the corresponding integral infinite.
\end{remark}

\begin{lemma}[Well-definedness]\label{lem:log-well-defined}
Let $\gamma$ be either a tame $C^1$ embedded closed curve or a nonconstant polygonal closed curve with finitely many edges.  Then $\log|x-y|$ and $|x-y|^p$, for every $p\in(-1,0)$, are integrable on $\gamma\times\gamma$ with respect to arc-length measure.  Consequently
\[
0<\sD_p(\gamma)<\infty
\qquad\text{for every }p\in(-1,\infty].
\]
\end{lemma}

\begin{proof}
For $p>0$ and $p=\infty$ there is nothing to prove, since the kernel is bounded and not identically zero.  Let $\widetilde\gamma\colon\RR/L\mathbb Z\to\RR^3$ be an arc-length parametrization and write $d_L(s,t)=\min_{k\in\mathbb Z}|s-t+kL|$ for the circular distance.

If $\gamma$ is a tame $C^1$ embedding, Lemma~\ref{lem:local-chord-arc} gives constants $c,\delta>0$ such that
\[
|\widetilde\gamma(s)-\widetilde\gamma(t)|\ge c\,d_L(s,t)
\]
whenever $d_L(s,t)<\delta$, while away from the circular diagonal the distance is bounded below by embeddedness.  Since $d_L^{\,p}$, for $p>-1$, and $|\log d_L|$ are integrable near the circular diagonal, the claim follows.

For a polygonal curve, split the double integral into finitely many pairs of edges.  Near an isolated intersection the distance is comparable to a two-dimensional linear singularity, and along overlapping collinear edges it reduces to a one-dimensional factor of the form $|s-t|$.  Thus $|x-y|^p$ is locally integrable for $p>-1$, and so is $\log|x-y|$.
\end{proof}

\begin{proposition}[The case \texorpdfstring{$p=0$}{p=0} as a limit]\label{prop:pzero-limit}
Let $\gamma$ be a tame $C^1$ embedded closed curve in $\RR^3$.  Then
\[
\lim_{p\to0}
\left(
\frac{1}{L^2}\int_\gamma\int_\gamma |x-y|^p\,ds_x\,ds_y
\right)^{1/p}
=
\exp\left(
\frac{1}{L^2}\int_\gamma\int_\gamma \log|x-y|\,ds_x\,ds_y
\right),
\]
where $L=\len(\gamma)$.  Thus the definition of $\sD_0$ is the continuous extension of $\sD_p$ at $p=0$.
\end{proposition}

\begin{proof}
Let $\mu=L^{-1}ds$ be normalized arc-length measure on $\gamma$, and set $f(x,y)=|x-y|$ on $\gamma\times\gamma$.  By Lemma~\ref{lem:log-well-defined}, $\log f\in L^1(\mu\times\mu)$.  We show that
\[
\lim_{p\to0}\frac{1}{p}\log\int f^p\,d\mu(x)d\mu(y)
=
\int\log f\,d\mu(x)d\mu(y).
\]
This is enough after exponentiating.

Choose $p_0\in(0,1)$.  Since $\gamma$ is compact, $f$ is bounded above.  Near the diagonal, the local chord-arc estimate of Lemma~\ref{lem:local-chord-arc} gives $f(x,y)\ge c\,d_L(s,t)$ in arc-length parameters.  Hence, for $|p|\le p_0$, the functions
\[
\left|\frac{f^p-1}{p}\right|
\]
are dominated near the diagonal by a constant multiple of
\[
1+d_L(s,t)^{-p_0}|\log d_L(s,t)|,
\]
which is integrable because $p_0<1$.  Away from the diagonal the domination is immediate.  Therefore dominated convergence gives
\[
\lim_{p\to0}\frac{\int f^p\,d\mu d\mu-1}{p}
=
\int\log f\,d\mu d\mu.
\]
The same domination applied to $f^p\to1$ gives $\int f^p\,d\mu d\mu\to1$.  The elementary relation $\log(1+a_p)\sim a_p$ now yields the displayed limit.
\end{proof}

\begin{proposition}\label{prop:scale-invariance-Dp}
For every $p\in(-1,\infty]$, the quantity $\sD_p(\gamma)$ scales linearly under Euclidean similarities.  Consequently, the ratio $\len(\gamma)/\sD_p(\gamma)$ is scale-invariant.
\end{proposition}

\begin{proof}
Let $\gamma_a$ be obtained from $\gamma$ by scaling with factor $a>0$.  Then $\len(\gamma_a)=a\len(\gamma)$.  For $p\ne0,\infty$, the double integral scales by $a^{p+2}$, while the normalizing factor $\len(\gamma_a)^2$ scales by $a^2$, so $\sD_p(\gamma_a)=a\sD_p(\gamma)$.  For $p=0$, the identity $\log|ax-ay|=\log a+\log|x-y|$ gives the same conclusion after exponentiating.  For $p=\infty$, it is immediate from the scaling of diameter.
\end{proof}

\begin{lemma}[$L^p$ convergence to the diameter]\label{lem:Lp-to-diameter}
Let $\gamma$ be a nonconstant rectifiable closed curve, equipped with normalized arc-length measure.  Then
\[
\lim_{p\to\infty}\sD_p(\gamma)=\diam(\gamma)=\sD_\infty(\gamma).
\]
\end{lemma}

\begin{proof}
On the probability space $\gamma\times\gamma$ with product normalized arc-length measure, $\sD_p(\gamma)$ is the $L^p$ norm of the bounded function
\[
f(x,y)=|x-y|.
\]
Its essential supremum equals its supremum.  Indeed, if $x_0,y_0$ are diametral points, then every pair of sufficiently small relative neighborhoods of $x_0$ and $y_0$ has positive arc-length measure, and $f$ is arbitrarily close to $\diam(\gamma)$ on their product.  The standard convergence of $L^p$ norms to the essential supremum on a probability space now gives the result.
\end{proof}

\begin{proposition}[Monotonicity in the exponent]\label{prop:p-monotonicity}
Let $-1<p<q\le\infty$.  Let $\gamma$ be either a tame $C^1$ embedded closed curve or a nonconstant polygonal closed curve with finitely many edges.  Then
\[
\sD_p(\gamma)\le\sD_q(\gamma).
\]
In particular, $\sD_p(\gamma)\le\diam(\gamma)$ for every $p\in(-1,\infty]$.
\end{proposition}

\begin{proof}
Let $\mu=L^{-1}ds$ be normalized arc-length measure and put $f(x,y)=|x-y|$.  For finite nonzero $r>-1$, set
\[
\Phi(r):=\log\int f^r\,d(\mu\times\mu).
\]
The function $\Phi$ is convex and $\Phi(0)=0$.  Hence the secant slope $\Phi(r)/r$ is non-decreasing in $r$, with its value at $r=0$ interpreted as
\[
\Phi'(0)=\int\log f\,d(\mu\times\mu).
\]
By Lemma~\ref{lem:log-well-defined}, $\Phi$ is finite on $(-1,\infty)$, and the dominated-convergence argument in the proof of Proposition~\ref{prop:pzero-limit}, which uses only the integrability of $f^{-p_0}|\log f|$ for some $p_0\in(0,1)$, justifies this interpretation in both cases.  Since
\[
\log\sD_r(\gamma)=\frac{\Phi(r)}{r}
\]
for $r\ne0$, this proves the assertion for finite exponents, including passage across $r=0$.  For $q=\infty$, apply the finite-exponent inequality with an exponent $r>p$ and let $r\to\infty$ using Lemma~\ref{lem:Lp-to-diameter}.
\end{proof}

\subsection{Continuous and polygonal unconstrained densities}

\begin{definition}
Let $K$ be a knot type and let $p\in(-1,\infty]$.  The \emph{unconstrained $p$-density} of $K$ is
\[
\rho_p(K):=
\inf_{\gamma\in\Hg(K)}
\frac{\len(\gamma)}{\sD_p(\gamma)}.
\]
\end{definition}

\begin{definition}\label{def:polygonal-knots}
For an integer $n\ge3$, let $\Poly_n(K)$ denote the set of embedded polygonal knots in $\RR^3$ with $n$ parametrized edge intervals and representing $K$.  A vertex with zero turning angle is allowed, so subdivision gives the natural inclusion
\[
\Poly_n(K)\subset\Poly_{n+1}(K).
\]
\end{definition}

\begin{remark}
Subdividing edges changes neither the image of a polygonal knot nor the arc-length measure on it, hence neither $\len$ nor $\sD_p$.  In particular, under this convention the number $e(K)$ defined below is still the usual stick number: zero-turning subdivision vertices may simply be deleted.
\end{remark}

\begin{definition}
Let $P\in\Poly_n(K)$.  Define $\sD_p(P)$ by the same formulas as in Definition~\ref{def:spread}, using uniform arc-length measure along the edges of $P$.  Define the \emph{polygonal unconstrained $p$-density} by
\[
\rho_{p,n}(K):=
\inf_{P\in\Poly_n(K)}
\frac{\len(P)}{\sD_p(P)}.
\]
\end{definition}

\begin{corollary}\label{cor:density-p-monotonicity}
If $-1<p<q\le\infty$, then
\[
\rho_p(K)\ge\rho_q(K),
\qquad
\rho_{p,n}(K)\ge\rho_{q,n}(K)
\]
whenever the polygonal density is defined.
\end{corollary}

\begin{proof}
By Proposition~\ref{prop:p-monotonicity}, every representative satisfies
\[
\frac{\len(\gamma)}{\sD_p(\gamma)}
\ge
\frac{\len(\gamma)}{\sD_q(\gamma)}.
\]
Take the infimum over the corresponding continuous or polygonal admissible class.
\end{proof}

\begin{lemma}\label{lem:polygonal-smoothing-ratio}
Let $P\in\Poly_n(K)$ and let $p\in(-1,\infty]$.  Then there exists a sequence of smooth representatives $P_j\in\Hg(K)$ such that
\[
\len(P_j)\to\len(P),
\qquad
\sD_p(P_j)\to\sD_p(P).
\]
\end{lemma}

\begin{proof}
Choose pairwise disjoint small balls around the vertices of $P$, disjoint from the rest of the polygon.  Inside each ball, replace the two incident straight subarcs by a smooth rounding that agrees to all orders with the adjacent straight subarcs near the boundary of the ball.  The roundings are chosen so that their total added length tends to zero as the radii of the balls tend to zero.  The resulting smooth curves $P_j$ are ambient isotopic to $P$, converge to $P$ in the Hausdorff metric, and satisfy $\len(P_j)\to\len(P)$.

Let $\eta_j$ and $\eta$ be the rescaled arc-length parametrizations of $P_j$ and $P$ on $S^1$.  Since $P$ has finitely many edges and all vertex angles are nonzero, it satisfies a circular local chord-arc estimate.  The chosen roundings may be made with a uniform local chord-arc estimate away from a set of shrinking parameter measure.  The continuity lemmas in the appendix then give $\sD_p(P_j)\to\sD_p(P)$ for finite $p$, and the case $p=\infty$ follows from Lemma~\ref{lem:hausdorff-diameter}.
\end{proof}

\subsection{Elementary bounds}

\begin{lemma}\label{lem:length-diameter}
For every rectifiable simple closed curve $\gamma\subset\RR^3$,
\[
\len(\gamma)\ge 2\diam(\gamma).
\]
\end{lemma}

\begin{proof}
Choose points $x,y\in\gamma$ such that $|x-y|=\diam(\gamma)$.  The curve $\gamma$ is divided by $x$ and $y$ into two subarcs, each of length at least $|x-y|$.  Hence $\len(\gamma)\ge2\diam(\gamma)$.
\end{proof}

\begin{proposition}\label{prop:basic-properties}
Let $K$ be a knot type and let $p\in(-1,\infty]$.  Then the following hold.
\begin{enumerate}
\item[(1)] Whenever $\Poly_n(K)\ne\varnothing$, both $\rho_p(K)$ and $\rho_{p,n}(K)$ are well defined.
\item[(2)] They satisfy $\rho_p(K)\ge2$ and $\rho_{p,n}(K)\ge2$.
\item[(3)] For every $n$ with $\Poly_n(K)\ne\varnothing$, one has $\rho_p(K)\le\rho_{p,n}(K)$.
\item[(4)] If $e(K):=\min\{n:\Poly_n(K)\ne\varnothing\}$, then $\{\rho_{p,n}(K)\}_{n\ge e(K)}$ is non-increasing.  Hence $\lim_{n\to\infty}\rho_{p,n}(K)$ exists.
\end{enumerate}
\end{proposition}

\begin{proof}
The ratios are well defined by Lemma~\ref{lem:log-well-defined}.  For continuous and polygonal representatives alike, Proposition~\ref{prop:p-monotonicity} and Lemma~\ref{lem:length-diameter} give
\[
\frac{\len(\gamma)}{\sD_p(\gamma)}
\ge
\frac{\len(\gamma)}{\diam(\gamma)}
\ge2,
\]
which proves (1) and (2).  To prove (3), fix $P\in\Poly_n(K)$.  Lemma~\ref{lem:polygonal-smoothing-ratio} provides representatives $P_j\in\Hg(K)$ whose ratios $\len(P_j)/\sD_p(P_j)$, each of which is at least $\rho_p(K)$, converge to $\len(P)/\sD_p(P)$.  Hence $\rho_p(K)\le\len(P)/\sD_p(P)$, and taking the infimum over $P$ gives (3).  Finally, subdivision of one edge sends $\Poly_n(K)$ to $\Poly_{n+1}(K)$ without changing length or $\sD_p$, so $\rho_{p,n+1}(K)\le\rho_{p,n}(K)$.
\end{proof}

\section{Degeneration of the unconstrained densities}\label{sec:degeneration}

\subsection{Local knotting}

\begin{lemma}[Uniform chord-arc control for scaled local patterns]\label{lem:scaled-pattern-chord-arc}
Let $\gamma$ be a tame $C^1$ embedded closed curve, and let $I_j\subset\gamma$ be subarcs whose lengths tend to zero.  Suppose that each $I_j$ is replaced by a scaled copy of one fixed embedded $C^1$ arc pattern $A$, with the same endpoint tangents and with straight collar subarcs near the two endpoints.  Assume that the scaling factors tend to zero and that the replacements are made in local balls meeting $\gamma$ exactly in the corresponding $I_j$.
Then, after choosing the local balls sufficiently small, the resulting curves satisfy a uniform local chord-arc estimate: there are constants $c>0$ and $\delta>0$, independent of $j$, such that, for all sufficiently large $j$,
\[
 |x-y|\ge c\, d_{\gamma_j}(x,y)
\]
whenever $x,y\in\gamma_j$ and their intrinsic distance $d_{\gamma_j}(x,y)$ is at most $\delta$.
\end{lemma}

\begin{proof}
There are three regions to control.
\begin{itemize}
\item \emph{Exterior.}  The original curve $\gamma$ is locally chord-arc by compactness and the $C^1$ embedding condition.  After shrinking the replacement balls, the same constants work on the exterior part of every $\gamma_j$.
\item \emph{Scaled model.}  The fixed pattern $A$, together with straight collar rays at its endpoints, is an embedded compact $C^1$ arc in a slightly larger model ball and therefore has a positive local chord-arc constant.  Scaling preserves the quotient of Euclidean and intrinsic distances.
\item \emph{Gluing neighborhoods.}  Choose each model collar to have length equal to a fixed positive fraction of the model locality radius.  After rescaling a replacement ball of radius $r_j$ by $r_j^{-1}$, the matching tangents and straight collars make the gluing neighborhoods uniformly $C^1$-close to a straight interval.  Their lower chord-arc constant therefore depends only on the fixed model and collar fraction.
\end{itemize}
Taking the minimum of these three constants gives $c$ and $\delta$ independent of $j$.
\end{proof}

\begin{lemma}[Continuity under a localized replacement]\label{lem:localized-replacement}
Let $\gamma$ be a tame $C^1$ embedded closed curve.  Let $I_j\subset\gamma$ be subarcs with $\ell_j=\len(I_j)\to0$, and suppose that $I_j$ is replaced by a $C^1$ embedded arc $A_j$ with the same endpoints and matching endpoint tangents.  Assume that the resulting closed curves $\gamma_j$ are embedded, that $\len(A_j)\to0$, and that $I_j\cup A_j$ is contained in a ball of radius $r_j\to0$.  Then, for every $p\in(-1,\infty]$,
\[
\len(\gamma_j)\to\len(\gamma)
\qquad\text{and}\qquad
\sD_p(\gamma_j)\to\sD_p(\gamma),
\]
provided the replacements are chosen from a fixed scaled $C^1$ pattern, so that the modified curves satisfy a uniform local chord-arc estimate.
\end{lemma}

\begin{proof}
The convergence of length follows from
\[
\len(\gamma_j)-\len(\gamma)=\len(A_j)-\len(I_j)\to0.
\]
For $p=\infty$, the images converge in the Hausdorff metric, so the diameters converge by Lemma~\ref{lem:hausdorff-diameter}.

It remains to consider finite $p$.  Write $E_j=I_j\cup A_j$ and let $C_j$ denote the common part of $\gamma$ and $\gamma_j$.  The contribution to the unnormalized double integrals from $C_j\times C_j$ is identical for $\gamma$ and $\gamma_j$.  Thus it is enough to show that all terms involving $I_j$ or $A_j$ tend to zero, after comparing the normalizing factors, which converge because the total lengths converge.

For $p>0$, this is immediate from the boundedness of the kernel and the fact that the arc-length measure of $I_j\cup A_j$ tends to zero.  For $-1<p<0$, the singularity requires a local estimate.  The uniform local chord-arc bound gives constants $c,\delta>0$, independent of $j$, such that on short parameter intervals
\[
|x-y|\ge c\,d_{\gamma_j}(x,y),
\]
where $d_{\gamma_j}$ denotes intrinsic distance along the modified curve.  Hence the self-interaction of the inserted arc satisfies
\[
\int_{A_j}\int_{A_j}|x-y|^p\,ds_x\,ds_y
\le
C\int_0^{a_j}\int_0^{a_j}|s-t|^p\,ds\,dt
=O(a_j^{p+2}),
\]
where $a_j=\len(A_j)\to0$.  The same estimate applies to $I_j\times I_j$.

For the cross terms, it is enough to estimate near the two endpoints; away from the endpoints the distance to the common part is bounded below and the contribution is $O(a_j+\ell_j)$.  Near an endpoint, using arc-length coordinates $s\in[0,a_j]$ on $A_j$ and $t\in[0,\delta]$ on the adjacent common arc, the chord-arc estimate gives
\[
|x(s)-y(t)|\ge c(s+t).
\]
Therefore
\[
\int_0^{a_j}\int_0^\delta |x(s)-y(t)|^p\,dt\,ds
\le
C\int_0^{a_j}\int_0^\delta (s+t)^p\,dt\,ds
=O(a_j)+O(a_j^{p+2})\to0.
\]
The corresponding estimates for $I_j$ are identical.  This proves convergence of the double integrals for $-1<p<0$.

For $p=0$, the same decomposition is applied to the logarithmic integral.  The self-interactions over $A_j\times A_j$ and $I_j\times I_j$ are bounded in absolute value by $O(a_j^2|\log a_j|)$ and $O(\ell_j^2|\log \ell_j|)$, respectively.  The endpoint cross terms are bounded by
\[
C\int_0^{a_j}\int_0^\delta |\log(s+t)|\,dt\,ds
=O(a_j|\log a_j|)+O(a_j)\to0,
\]
with the same estimate for $I_j$.  Away from the endpoints the logarithmic kernel is bounded, and the cross contribution is again $O(a_j+\ell_j)$.  Thus the logarithmic double integrals converge.  Exponentiating gives convergence of $\sD_0$.
\end{proof}

\begin{theorem}[Local knotting degeneration]\label{thm:local-knotting}
Let $p\in(-1,\infty]$.  For all knot types $K$ and $J$,
\[
\rho_p(K\mathbin{\#}J)\le\rho_p(K).
\]
In particular,
\[
\rho_p(K)\le\rho_p(U),
\]
where $U$ denotes the unknot.
More precisely, any representative of $K$ may be modified inside an arbitrarily small ball so as to represent $K\mathbin{\#}J$, while changing both length and $\sD_p$ by an arbitrarily small amount.
\end{theorem}

\begin{proof}
Let $\gamma\in\Hg(K)$ and fix a small subarc $I\subset\gamma$ contained in a ball $B$ whose intersection with $\gamma$ is exactly $I$.  Replace $I$ by a properly embedded knotted arc in $B$ with the same endpoints and matching tangent directions, chosen so that closing this arc by the complementary trivial boundary arc realizes $J$.  Scaling the pattern inside $B$ and taking $B$ sufficiently small, the replacement may be arranged so that the additional length is arbitrarily small.  The resulting curve, after a harmless $C^1$ smoothing near the endpoints, represents $K\mathbin{\#}J$.

The replacement can be taken from a fixed $C^1$ knotted-arc pattern with straight collars near its endpoints and then scaled into $B$.  Hence the lengths of both the removed and inserted arcs tend to zero.  By Lemma~\ref{lem:scaled-pattern-chord-arc}, the modified curves satisfy the uniform local chord-arc estimate required in Lemma~\ref{lem:localized-replacement}.  Therefore $\sD_p$ converges to $\sD_p(\gamma)$ for every $p\in(-1,\infty]$.

Thus for every $\eps>0$ there is a representative $\gamma_\eps$ of $K\mathbin{\#}J$ with
\[
\frac{\len(\gamma_\eps)}{\sD_p(\gamma_\eps)}
<
\frac{\len(\gamma)}{\sD_p(\gamma)}+\eps.
\]
Taking the infimum over $\gamma\in\Hg(K)$ gives $\rho_p(K\mathbin{\#}J)\le\rho_p(K)$.  Setting $K=U$ and using $U\mathbin{\#}J=J$ gives the stated special case.
\end{proof}

\begin{figure}[H]
\centering
\begin{tikzpicture}[>=Stealth,font=\small]
\draw[thick] plot[smooth cycle,tension=0.8] coordinates {(-2.6,0.2) (-1.6,1.3) (0,0.9) (1.5,1.5) (2.6,0.3) (1.6,-1.0) (0,-1.2) (-1.7,-0.9)};
\draw[gray!70!black] (0,-1.2) circle (0.22);
\node at (-1.3,0.2) {$\gamma$};
\node[below] at (0,-1.45) {$B$};
\draw[gray,densely dashed] (0.1,-0.99)--(5.45,1.43);
\draw[gray,densely dashed] (0.1,-1.41)--(5.45,-1.83);
\begin{scope}[shift={(6.3,0)},scale=0.4]
\draw[gray!70!black] (0,-0.75) circle (4.75);
\draw[gray,densely dashed] (-3.5,-3)--(3.5,-3);
\node[gray!50!black,below] at (0,-3.05) {$I_j$};
\draw[thick] (-4.18,-3)--(-3.5,-3) .. controls (-2.2,-3) and (-1.327,-2.6) .. (-1.327,-1.156);
\draw[thick] (4.18,-3)--(3.5,-3) .. controls (2.2,-3) and (1.327,-2.6) .. (1.327,-1.156);
\draw[thick,domain=3.8306:8.7358,samples=140,variable=\t,smooth] plot ({sin(\t r)+2*sin(2*\t r)},{cos(\t r)-2*cos(2*\t r)});
\foreach \a/\b in {4.31/4.61,6.404/6.704,8.5/8.7358}{
\draw[white,line width=5.5pt,domain=\a:\b,samples=20,variable=\t,smooth] plot ({sin(\t r)+2*sin(2*\t r)},{cos(\t r)-2*cos(2*\t r)});
\draw[thick,domain=\a-0.06:\b,samples=24,variable=\t,smooth] plot ({sin(\t r)+2*sin(2*\t r)},{cos(\t r)-2*cos(2*\t r)});
}
\node at (0,3.2) {$A_j$};
\end{scope}
\draw[<->] (6.3,-2.45)--(8.2,-2.45) node[midway,below]{$r_j\to0$};
\end{tikzpicture}
\caption{Local knotting.  A subarc $I_j$ of $\gamma$ inside a ball $B$ of radius $r_j$ is replaced by a scaled copy $A_j$ of a fixed knotted arc of type $J$ with straight collars.  Both $\len$ and $\sD_p$ change by $o(1)$ as $r_j\to0$.}
\label{fig:local-knotting}
\end{figure}

\begin{remark}
The theorem shows that the unconstrained $p$-density cannot be expected to give lower bounds for knot complexity.  Any complexity that can be localized at arbitrarily small scale is invisible to the infimum.
\end{remark}

\subsection{The mean-chord range \texorpdfstring{$-1<p\le2$}{-1<p<=2}}

We next record the sharp mean-chord input used in the finite range up to $p=2$.  The circle extremality of the mean chord at $p=1$ goes back to L\"uk\H{o} \cite{Luko}; the exponent range and strict equality statements used here follow from Exner--Harrell--Loss.  Let $C_L$ denote a round circle of length $L$.

\begin{theorem}[Mean-chord inequality]\label{thm:mean-chord}
Let $\gamma$ be a nonconstant rectifiable closed curve in $\RR^3$ of length $L$, understood through an arc-length parametrization, and let $C_L$ be a round circle of the same length.  For every $p\in(-1,2]$,
\[
\sD_p(\gamma)\le \sD_p(C_L).
\]
Whenever $\sD_p(\gamma)>0$, equality holds if and only if the arc-length parametrized image of $\gamma$ is a round circle.
\end{theorem}

\begin{proof}
For $p\in(-1,0)\cup(0,2]$, the inequality $M_p(\gamma)\le M_p(C_L)$ for $0<p\le2$ and its reverse for $-1<p<0$, where $M_p$ denotes the normalized double integral of $|x-y|^p$, are the integrated $L^p$-mean chord inequalities of Exner--Harrell--Loss \cite[Proposition~2.1 and Theorem~2.2]{EHL}.  Their Fourier argument is formulated for unit-speed loops and applies to the Lipschitz arc-length parametrization of a rectifiable curve.  The one-page addendum \cite{EHLAddendum} corrects the final nonsmooth equality argument in the proof of their Theorem~2.2.  The resulting strict equality statement shows that equality can occur only for a planar round circle.  For $p<0$, taking the $1/p$-th power reverses the moment inequality, so in both signs of $p$ one obtains $\sD_p(\gamma)\le\sD_p(C_L)$ with the asserted equality condition.  If a negative moment is infinite, then $\sD_p(\gamma)=0$ and the inequality is automatic.

The logarithmic case requires a separate strictness argument.  If the logarithmic double integral is $-\infty$, then $\sD_0(\gamma)=0$ and the inequality is immediate.  Otherwise, let
\[
c\colon\RR/L\mathbb Z\longrightarrow\RR^3
\]
be an arc-length parametrization.  For $0<u<L$, write
\[
d_u(s):=|c(s+u)-c(s)|.
\]
Jensen's inequality and the $p=2$ mean-chord inequality give
\[
\frac1L\int_0^L\log d_u(s)\,ds
\le
\frac12\log\left(
\frac1L\int_0^Ld_u(s)^2\,ds
\right)
\le
\log\left(
\frac{L}{\pi}\sin\frac{\pi u}{L}
\right)
\]
for $0<u<L$, with the sine interpreted symmetrically about $L/2$.  Integrating in $u$ and using Fubini's theorem yields
\[
\log\sD_0(\gamma)
\le
\log\sD_0(C_L).
\]
If equality holds, then the second inequality is an equality for almost every $u$.  The equality condition in the $p=2$ mean-chord theorem forces $\gamma$ to be a round circle.  Conversely, a round circle gives equality throughout.
\end{proof}

\begin{lemma}[Round-circle values]\label{lem:round-circle-values}
For a round circle $C_L$ of length $L$, one has, for $p\ne0$,
\[
\sD_p(C_L)
=
\frac{L}{\pi}
\left(
\frac{1}{\pi}\int_0^\pi \sin^p\theta\,d\theta
\right)^{1/p}.
\]
Moreover,
\[
\sD_0(C_L)=\frac{L}{2\pi}.
\]
\end{lemma}

\begin{proof}
Let $\widetilde C_L\colon\RR/L\mathbb Z\to\RR^3$ be an arc-length parametrization of the circle.  If the circular arc-length distance between two parameters is $u\in[0,L/2]$, then
\[
 |\widetilde C_L(s+u)-\widetilde C_L(s)|=
 \frac{L}{\pi}\sin\frac{\pi u}{L}.
\]
Because the integrand depends only on the circular difference of the two parameters, the normalized double integral reduces to
\[
\begin{aligned}
\sD_p(C_L)^p
&=
\frac{1}{L^2}\int_0^L\int_0^L
 |\widetilde C_L(s)-\widetilde C_L(t)|^p\,ds\,dt \\
&=
\frac{2}{L}\int_0^{L/2}
 \left(\frac{L}{\pi}\sin\frac{\pi u}{L}\right)^p du.
\end{aligned}
\]
Indeed, for each fixed $s$, the two points at circular distance $u\in(0,L/2)$ from $s$ give the factor $2$, while the endpoints $u=0$ and $u=L/2$ have measure zero.  With the change of variables $\theta=\pi u/L$, this becomes
\[
\sD_p(C_L)^p
=
\left(\frac{L}{\pi}\right)^p
\frac{2}{\pi}\int_0^{\pi/2}\sin^p\theta\,d\theta
=
\left(\frac{L}{\pi}\right)^p
\frac{1}{\pi}\int_0^\pi\sin^p\theta\,d\theta.
\]
Taking the $1/p$-th power gives the stated formula.  For $-1<p<0$, the integral $\int_0^\pi\sin^p\theta\,d\theta$ is finite because $\sin\theta$ is comparable to the distance from $\theta$ to $\{0,\pi\}$ and $p>-1$.  For $p=0$, either take the limit $p\to0$, or compute directly using
\[
\frac{1}{\pi}\int_0^\pi \log(\sin\theta)\,d\theta=-\log2.
\]
This gives $\sD_0(C_L)=L/(2\pi)$.
\end{proof}

For later use, define the \emph{circle density constant} by
\begin{equation}\label{eq:circle-density-constant}
c_p:=
\begin{cases}
\displaystyle
\pi\left(
\dfrac{\pi}{\int_0^\pi\sin^p\theta\,d\theta}
\right)^{1/p},&p\ne0,\\[1.2ex]
2\pi,&p=0.
\end{cases}
\end{equation}
Thus $c_p=\len(C_L)/\sD_p(C_L)$ and, in particular, $c_2=\sqrt2\pi$.

\begin{theorem}[Complete degeneration in the mean-chord range]\label{thm:mean-chord-degeneration}
For every knot type $K$ and every $p\in(-1,2]$,
\[
\rho_p(K)=c_p.
\]
\end{theorem}

\begin{proof}
By Theorem~\ref{thm:mean-chord} and Lemma~\ref{lem:round-circle-values}, every $\gamma\in\Hg(K)$ of length $L$ satisfies $\sD_p(\gamma)\le\sD_p(C_L)=L/c_p$, so $\rho_p(K)\ge c_p$.  A round circle represents $U$ and attains this bound, so $\rho_p(U)=c_p$.  Theorem~\ref{thm:local-knotting} gives $\rho_p(K)\le\rho_p(U)$, and the two bounds coincide.
\end{proof}

\begin{corollary}[Attainment in the mean-chord range]\label{cor:mean-chord-attainment}
Let $-1<p\le2$.  The infimum defining $\rho_p(K)$ is attained if and only if $K=U$.  In the unknot class, every minimizer is a round circle up to similarities and reparametrization.
\end{corollary}

\begin{proof}
A round circle attains the value $c_p$ for the unknot.  Conversely, if a representative of $K$ attains $\rho_p(K)=c_p$, then equality holds in Theorem~\ref{thm:mean-chord}.  Its equality statement forces that representative to be a round circle, so $K=U$.
\end{proof}

\subsection{Radius-of-gyration interpretation at \texorpdfstring{$p=2$}{p=2}}

\begin{definition}
Let $\gamma\in\Hg(K)$ be parametrized by arc length on $[0,L]$.  Its center of mass is
\[
c(\gamma):=\frac1L\int_0^L\gamma(s)\,ds,
\]
and its radius of gyration is
\[
\Rg(\gamma):=
\left(
\frac1L\int_0^L |\gamma(s)-c(\gamma)|^2\,ds
\right)^{1/2}.
\]
\end{definition}

\begin{proposition}\label{prop:p=2}
For every rectifiable closed curve $\gamma$,
\[
\sD_2(\gamma)=\sqrt2\,\Rg(\gamma).
\]
\end{proposition}

\begin{proof}
Let $L=\len(\gamma)$ and $c=c(\gamma)$.  The standard second-moment identity gives
\[
\frac{1}{L^2}\int_\gamma\int_\gamma |x-y|^2\,ds_x\,ds_y
=
\frac2L\int_\gamma |x-c|^2\,ds
=2\Rg(\gamma)^2.
\]
Taking square roots proves the claim.
\end{proof}

Thus $\rho_2$ is the length-to-radius-of-gyration density, up to the factor $\sqrt2$, and Theorem~\ref{thm:mean-chord-degeneration} gives $\rho_2(K)=\sqrt2\pi$.

\subsection{The diameter endpoint}

\begin{theorem}[Complete degeneration at $p=\infty$]\label{thm:infty-degeneration}
For every knot type $K$,
\[
\rho_\infty(K)=2.
\]
Moreover, this infimum is not attained by any embedded $C^1$ simple closed curve.
\end{theorem}

\begin{proof}
Lemma~\ref{lem:length-diameter} gives $\rho_\infty(K)\ge2$.  For the reverse inequality it suffices, by Theorem~\ref{thm:local-knotting}, to show $\rho_\infty(U)\le2$.  For $R,\eps>0$, let $\sigma_{R,\eps}$ be the planar stadium curve bounding the closed $\eps$-neighborhood, in a plane, of a segment of length $2R$; see Figure~\ref{fig:stadium}.  It is a $C^1$ embedded unknot with
\[
\len(\sigma_{R,\eps})=4R+2\pi\eps,
\qquad
\diam(\sigma_{R,\eps})=2R+2\eps,
\]
so its ratio tends to $2$ as $\eps/R\to0$.  Thus $\rho_\infty(U)\le2$, and hence $\rho_\infty(K)=2$.

It remains to justify non-attainment.  If equality were attained by an embedded $C^1$ simple closed curve, choose a pair of diametral points $x,y$.  The two subarcs between $x$ and $y$ would each have length at least $|x-y|$.  Since the total length would be $2|x-y|$, both subarcs would have length exactly $|x-y|$.  A rectifiable arc whose length equals the Euclidean distance between its endpoints is the straight segment between those endpoints.  Hence both subarcs would be the same straight segment from $x$ to $y$, producing a doubled segment rather than an embedded simple closed curve.
\end{proof}

\begin{figure}[H]
\centering
\begin{tikzpicture}[>=Stealth,font=\small]
\draw[thick,rounded corners=0.3cm] (-4.5,-0.3) rectangle (4.5,0.3);
\draw[gray,densely dashed] (-4.2,0)--(4.2,0);
\fill[white] (1.5,0.3) circle (0.2);
\draw[gray!70!black] (1.5,0.3) circle (0.2);
\begin{scope}[shift={(1.5,0.33)},scale=0.045]
\draw[semithick,domain=0:6.2832,samples=80,variable=\t,smooth cycle] plot ({sin(\t r)+2*sin(2*\t r)},{cos(\t r)-2*cos(2*\t r)});
\end{scope}
\node[above] at (1.5,0.5) {$J$};
\draw[<->] (-4.2,-0.65)--(4.2,-0.65) node[midway,below]{$2R$};
\draw[<->] (4.85,-0.3)--(4.85,0.3) node[midway,right]{$2\eps$};
\node[left] at (-4.6,0) {$\sigma_{R,\eps}$};
\end{tikzpicture}
\caption{A thin stadium curve with a knot of type $J$ tied in a small ball.  As $\eps/R\to0$ the ratio $\len/\diam$ tends to $2$, and the curve degenerates to a doubled segment.}
\label{fig:stadium}
\end{figure}

\subsection{The finite high-exponent range \texorpdfstring{$p>2$}{p>2}}

The round-circle formula from the mean-chord range should not be extended to finite exponents $p>2$.  The knot-type question can nevertheless be settled without identifying the maximizing curve.  For a knot type $K$, set
\[
A_p(K):=
\sup_{\gamma\in\Hg(K)}
\frac{\sD_p(\gamma)}{\len(\gamma)},
\qquad
\rho_p(K)=\frac{1}{A_p(K)},
\]
and let
\[
A_p^{\mathrm{all}}:=
\sup_{\gamma}
\frac{\sD_p(\gamma)}{\len(\gamma)},
\]
where the latter supremum is taken over all tame $C^1$ embedded closed curves.

We use the following geometric convexification theorem of Sallee.  The formulation below is also recorded in the proof of Proposition~5.2 of Abrams--Cantarella--Fu--Ghomi--Howard \cite{ACFGH}.

\begin{theorem}[Sallee's chord-stretching theorem \cite{Sallee}]\label{thm:sallee-stretching}
Let $c\colon\RR/L\mathbb Z\to\RR^3$ be a closed unit-speed space curve of length $L$.  There is a closed convex unit-speed plane curve
\[
c^*\colon\RR/L\mathbb Z\to\RR^2
\]
of the same length such that
\[
|c(s)-c(t)|\le |c^*(s)-c^*(t)|
\qquad
\text{for all }s,t\in\RR/L\mathbb Z.
\]
\end{theorem}

\begin{remark}\label{rem:sallee-all-exponents}
The pointwise chord inequality has the same effect on $\sD_p$ for every $p\in(-1,\infty]$: it increases positive moments, decreases negative moments before the negative $1/p$ power reverses the order, increases the logarithmic mean, and increases the diameter.  Thus Sallee's theorem gives a formal convexification principle throughout the full exponent range.  We use it below only for $p>2$.  In the nonpositive range, a degenerate convex output requires an additional uniform-integrability argument along the second diagonal of the limiting doubled segment, whereas the Exner--Harrell--Loss theorem already supplies the sharp constant and equality case directly.
\end{remark}

For the supremum over smooth unknots, it is harmless if the resulting convex plane curve has corners or is degenerate.  Such a curve can be approximated uniformly, with convergent lengths and arc-length parametrizations, by smooth strictly convex plane curves.  After rescaling, the approximating curves may be taken to have length exactly $L$.

\begin{theorem}[Complete degeneration in the finite high-exponent range]\label{thm:finite-high-degeneration}
Let $K$ be a knot type and let $2<p<\infty$.  Then
\[
A_p(K)=A_p(U)=A_p^{\mathrm{all}}.
\]
Equivalently,
\[
\rho_p(K)=\rho_p(U)=\frac{1}{A_p^{\mathrm{all}}}.
\]
In particular, the unconstrained $p$-density is independent of the knot type throughout the finite high-exponent range.
\end{theorem}

\begin{proof}
Let $\gamma\in\Hg(K)$ have length $L$, and let $c$ be its arc-length parametrization.  Apply Theorem~\ref{thm:sallee-stretching} to obtain a convex plane curve $c^*$ with the same length and with no shorter corresponding chords.  Since $p>0$,
\[
\begin{aligned}
\sD_p(\gamma)^p
&=\frac{1}{L^2}\int_0^L\int_0^L
 |c(s)-c(t)|^p\,ds\,dt \\
&\le
\frac{1}{L^2}\int_0^L\int_0^L
 |c^*(s)-c^*(t)|^p\,ds\,dt
=\sD_p(c^*)^p.
\end{aligned}
\]
A smooth strictly convex plane curve is an unknot.  If $c^*$ is not smooth and strictly convex, approximate it by smooth strictly convex curves as described above.  Lemma~\ref{lem:Dp-continuity}(a) applies even when the limiting convex curve is degenerate and gives convergence of $\sD_p$.  Hence
\[
\frac{\sD_p(\gamma)}{L}
\le A_p(U).
\]
Taking the supremum over $\gamma\in\Hg(K)$ yields
\[
A_p(K)\le A_p(U).
\]
On the other hand, Theorem~\ref{thm:local-knotting} gives $\rho_p(K)\le\rho_p(U)$, or equivalently
\[
A_p(K)\ge A_p(U).
\]
Thus $A_p(K)=A_p(U)$.

The same chord-stretching and approximation argument applies to every tame $C^1$ embedded closed curve, irrespective of its knot type.  It follows that $A_p^{\mathrm{all}}\le A_p(U)$, while the reverse inequality is immediate from the definition.  Therefore
\[
A_p(K)=A_p(U)=A_p^{\mathrm{all}},
\]
and taking reciprocals completes the proof.
\end{proof}

\begin{corollary}[Complete unconstrained degeneration]\label{cor:complete-unconstrained-degeneration}
For every knot type $K$ and every $p\in(-1,\infty]$,
\[
\rho_p(K)=\rho_p(U).
\]
\end{corollary}

\begin{proof}
Combine Theorems~\ref{thm:mean-chord-degeneration}, \ref{thm:finite-high-degeneration}, and \ref{thm:infty-degeneration}.
\end{proof}

\begin{remark}[Attainment for finite high exponents]\label{rem:high-p-attainment}
Let $2<p<\infty$.  If some $\gamma\in\Hg(K)$ attains the common value $\rho_p(K)$, then $K=U$.  Indeed, apply Sallee's theorem to an arc-length parametrization $c$ of $\gamma$.  The convex output $c^*$ has no shorter corresponding chord.  Approximation of $c^*$ by smooth strictly convex curves shows that $\sD_p(c^*)/\len(c^*)$ cannot exceed the global supremum attained by $\gamma$.  Hence equality holds in the integrated pointwise chord inequality.  Its nonnegative continuous integrand must vanish identically, so
\[
|c(s)-c(t)|=|c^*(s)-c^*(t)|
\qquad\text{for all }s,t.
\]
Equality of all pairwise distances makes $c$ congruent to the convex planar curve $c^*$; in particular, $\gamma$ is an unknot.

Conversely, if the convex maximizer furnished by \cite[Proposition~5.2]{ACFGH} is a nondegenerate admissible oval, then it represents $U$ and attains $\rho_p(U)$.  Under this condition, therefore, the infimum is attained if and only if $K=U$.  If the convex maximizer is degenerate or fails to belong to the admissible $C^1$ class, attainment even for $U$ requires a separate regularity analysis.
\end{remark}

\begin{remark}\label{rem:sharp-planar-problem}
Theorem~\ref{thm:finite-high-degeneration} settles knot-type independence but does not give a closed formula for the common value when $2<p<\infty$.  Proposition~5.2 of \cite{ACFGH} shows that the normalized $L^p$-mean chord functional has a maximizer in the compact class of convex plane curves.  Determining this maximizer is the sharp extremal problem posed in Question~5.1 of that paper.  Section~5 of \cite{ACFGH} also contains numerical experiments on the loss of circle optimality and on convergence toward a doubly covered segment.  The following elementary comparison makes one part of that discussion rigorous and identifies the number reported there as approximately $3.5721$.
\end{remark}

\begin{proposition}[Doubled-segment obstruction]\label{prop:doubled-segment-threshold}
Let
\[
I_p:=\int_0^\pi\sin^p\theta\,d\theta
\]
and let $p_*>2$ be the unique solution of
\begin{equation}\label{eq:circle-segment-threshold}
\pi\left(\frac{\pi}{I_p}\right)^{1/p}
=
2\left(\frac{(p+1)(p+2)}{2}\right)^{1/p}.
\end{equation}
Then
\[
p_*=3.5720244135\ldots.
\]
For every $p>p_*$, the round circle is not a global maximizer of $\sD_p/\len$ among closed curves.  More precisely,
\[
\rho_p(U)
\le
d_p
:=
2\left(\frac{(p+1)(p+2)}{2}\right)^{1/p}
<c_p.
\]
\end{proposition}

\begin{proof}
Let $S_L$ be the degenerate closed curve of total length $L$ that traverses a line segment of length $L/2$ twice, once in each direction.  Normalized arc-length measure on $S_L$ pushes forward to uniform measure on that segment.  Hence, for $p>0$,
\[
\sD_p(S_L)^p
=
\frac{2}{(p+1)(p+2)}\left(\frac{L}{2}\right)^p,
\]
and therefore
\[
\frac{L}{\sD_p(S_L)}=d_p.
\]
By Lemma~\ref{lem:round-circle-values}, the circle density is $c_p$.  Thus equality of the two densities is exactly equation~\eqref{eq:circle-segment-threshold}.

To see uniqueness, set
\[
R(p):=
\left(\frac{\sD_p(S_L)}{\sD_p(C_L)}\right)^p
=
\frac{\pi^{p+1}2^{1-p}}{(p+1)(p+2)I_p}.
\]
Using
\[
I_p=\sqrt\pi\,
\frac{\Gamma((p+1)/2)}{\Gamma((p+2)/2)},
\]
and writing $\psi$ for the digamma function, differentiation gives
\[
\frac{d}{dp}\log R(p)
=
\log\frac{\pi}{2}
-\frac1{p+1}-\frac1{p+2}
+\frac12\left[
\psi\left(\frac{p+2}{2}\right)
-\psi\left(\frac{p+1}{2}\right)
\right].
\]
For $x>0$,
\[
\psi\left(x+\frac12\right)-\psi(x)
=
\int_0^\infty
\frac{e^{-xt}}{1+e^{-t/2}}\,dt
>
\frac1{2x}.
\]
With $x=(p+1)/2$, it follows for $p\ge2$ that
\[
\frac{d}{dp}\log R(p)
>
\log\frac{\pi}{2}
-\frac1{2(p+1)}-\frac1{p+2}
\ge
\log\frac{\pi}{2}-\frac16-\frac14
>0.
\]
Moreover, $R(2)=\pi^2/12<1$, while $R(p)\to\infty$ as $p\to\infty$.  Hence the solution is unique.  Numerical solution of the explicit monotone equation gives the stated value of $p_*$.

Finally, smooth convex stadium curves of length $L$ converge, in normalized arc-length parametrization, to $S_L$.  Lemma~\ref{lem:Dp-continuity}(a) gives convergence of their $\sD_p$ values.  For $p>p_*$ one has $d_p<c_p$, so a sufficiently thin smooth stadium has larger $\sD_p/\len$ than the circle.  This proves the strict obstruction and the bound on $\rho_p(U)$.
\end{proof}

\begin{remark}\label{rem:circle-breaking-bound}
The doubled-segment comparison does not determine the true circle-breaking exponent.  Numerical experiments in \cite[Section~5]{ACFGH} give the interval
\[
3.3<\alpha<3.5721
\]
for the loss of circle optimality.  The upper endpoint reported there is the rounded value of the explicit number $p_*$ in Proposition~\ref{prop:doubled-segment-threshold}.  The fixed-arc-length chord problem of \cite{EFH}, whose critical exponent at half-length separation is $5/2$, is related but distinct from the integrated double-chord functional considered here.
\end{remark}

The same direct integration identifies $d_p$ with the doubled-segment density for all $-1<p<\infty$, with the continuous value $d_0=2e^{3/2}$.  Moreover
\[
c_p\sim\frac{2}{p+1},
\qquad
d_p\sim\frac{4}{p+1}
\qquad (p\downarrow-1),
\]
so the doubled-segment density diverges twice as fast at the left endpoint; see Figure~\ref{fig:circle-segment-density}.

\begin{figure}[H]
\centering
\begin{minipage}{0.49\textwidth}
\centering
\begin{tikzpicture}
\begin{axis}[
  width=0.94\linewidth,
  height=0.68\linewidth,
  xlabel={$p$},
  ylabel={density},
  xmin=-0.9, xmax=10,
  ymin=1.8, ymax=55,
  ymode=log,
  ytick={2,3.14159265,6.28318531,10,20,50},
  yticklabels={$2$,$\pi$,$2\pi$,$10$,$20$,$50$},
  grid=major,
  tick label style={font=\scriptsize},
  label style={font=\small},
  legend style={font=\scriptsize,draw=none,fill=none,at={(0.98,0.98)},anchor=north east},
]
\addplot[blue,thick,smooth] coordinates {
(-0.9,26.42091386) (-0.75,13.37020127) (-0.5,8.75375846)
(-0.25,7.12871657) (0,6.28318531) (0.25,5.75908691)
(0.5,5.39975831) (1,4.93480220) (1.5,4.64389595)
(2,4.44288294) (2.5,4.29471414) (3,4.18042322)
(3.3,4.12343096) (3.5720244135,4.07759433) (4,4.01459792)
(5,3.89918735) (6,3.81365617) (8,3.69438892) (10,3.61442997)
};
\addlegendentry{$c_p$}
\addplot[orange!85!black,thick,smooth] coordinates {
(-0.9,50.19113391) (-0.75,23.76493483) (-0.5,14.22222222)
(-0.25,10.78332588) (0,8.96337814) (0.25,7.82132149)
(0.5,7.03125000) (1,6.00000000) (1.5,5.34993740)
(2,4.89897949) (2.5,4.56594021) (3,4.30886938)
(3.3,4.18067550) (3.5720244135,4.07759433) (4,3.93597934)
(5,3.67683257) (6,3.48516225) (8,3.21870786) (10,3.04077572)
};
\addlegendentry{$d_p$}
\addplot[gray,densely dashed] coordinates {(-0.9,3.14159265) (10,3.14159265)};
\addplot[gray,densely dashed] coordinates {(-0.9,2) (10,2)};
\end{axis}
\end{tikzpicture}
\end{minipage}
\hfill
\begin{minipage}{0.49\textwidth}
\centering
\begin{tikzpicture}
\begin{axis}[
  width=0.94\linewidth,
  height=0.68\linewidth,
  xlabel={$p$},
  xmin=2, xmax=5,
  ymin=3.55, ymax=5.02,
  grid=major,
  tick label style={font=\scriptsize},
  label style={font=\small},
]
\addplot[blue,thick,smooth] coordinates {
(2,4.44288294) (2.5,4.29471414) (3,4.18042322)
(3.3,4.12343096) (3.5720244135,4.07759433)
(4,4.01459792) (5,3.89918735)
};
\addplot[orange!85!black,thick,smooth] coordinates {
(2,4.89897949) (2.5,4.56594021) (3,4.30886938)
(3.3,4.18067550) (3.5720244135,4.07759433)
(4,3.93597934) (5,3.67683257)
};
\addplot[gray,densely dashed] coordinates {(3.5720244135,3.55) (3.5720244135,5.02)};
\addplot[black,only marks,mark=*,mark size=1.6pt] coordinates {(3.5720244135,4.07759433)};
\node[anchor=south west,font=\scriptsize] at (axis cs:3.61,4.09) {$p_*$};
\end{axis}
\end{tikzpicture}
\end{minipage}
\caption{The round-circle density $c_p$ and doubled-segment density $d_p$.  The overview shows divergence as $p\downarrow-1$ and the distinct limits $\pi$ and $2$ as $p\to\infty$.  The right panel magnifies the crossing at $p_*=3.5720244135\ldots$.}
\label{fig:circle-segment-density}
\end{figure}

\section{Ropelength-filtered densities}\label{sec:rop-window}

\subsection{Thickness and ropelength}

\begin{definition}
Let $\mathcal C^{1,1}(K)$ denote the embedded $C^{1,1}$ representatives of $K$.  For $\gamma\in\mathcal C^{1,1}(K)$, let $\Thi(\gamma)$ denote its thickness, equivalently the supremal radius of an embedded normal tube about $\gamma$.  Define
\[
\Rop(\gamma):=\frac{\len(\gamma)}{\Thi(\gamma)},
\qquad
\Rop(K):=
\inf_{\gamma\in\mathcal C^{1,1}(K)}\Rop(\gamma).
\]
After scaling, this is equivalently the infimum of $\len(\gamma)$ among representatives satisfying $\Thi(\gamma)\ge1$.
\end{definition}

\begin{remark}
Global radius and normal injectivity give equivalent formulations of thickness at the regularity used here; see \cite{GonzalezMaddocks,LSDR}.  The compactness theories of \cite{GMSvdM,CKS} establish the existence of ideal representatives; see also Rawdon~\cite{Rawdon}.
\end{remark}

\begin{lemma}\label{lem:unknot-ropelength}
Every knot type satisfies
\[
\Rop(K)\ge2\pi,
\]
with equality if and only if $K=U$.  Moreover, if $\Thi(\gamma)\ge1$ and $\len(\gamma)=2\pi$, then $\gamma$ is a unit circle up to a rigid motion.
\end{lemma}

\begin{proof}
The thickness bound implies $|\kappa|\le1$ almost everywhere.  Fenchel's theorem for $C^{1,1}$ closed curves therefore gives
\[
2\pi
\le
\operatorname{TC}(\gamma)
=
\int_\gamma|\kappa|\,ds
\le
\len(\gamma).
\]
A unit circle has thickness one and length $2\pi$, proving $\Rop(U)=2\pi$.  If equality holds throughout, the equality case of Fenchel's theorem makes $\gamma$ a convex planar curve, and $|\kappa|=1$ almost everywhere.  Hence it is a unit circle.  Since ropelength minimizers exist, equality $\Rop(K)=2\pi$ for a knot type is realized by such a normalized curve and therefore forces $K=U$.
\end{proof}

\subsection{Definition of the filtered invariant}

\begin{definition}\label{def:rop-filtered-density}
Let $K$ be a knot type, $p\in(-1,\infty]$, and $\lambda\ge1$.  Define the \emph{ropelength-filtered $p$-density at level $\lambda$} by
\[
\rho^{\operatorname{rop}}_{p,\lambda}(K):=
\inf\left\{
\frac{\len(\gamma)}{\sD_p(\gamma)}
\ \middle|\
\gamma\in\mathcal C^{1,1}(K),\
\Rop(\gamma)\le \lambda\Rop(K)
\right\}.
\]
The admissible class in this infimum is called the \emph{ropelength window} of level $\lambda$.
\end{definition}

\begin{remark}
This formulation is scale-invariant.  Scaling each admissible curve to thickness one gives the equivalent normalized definition
\[
\rho^{\operatorname{rop}}_{p,\lambda}(K)
=
\inf\left\{
\frac{\len(\gamma)}{\sD_p(\gamma)}
\ \middle|\
\begin{gathered}
\gamma\in\mathcal C^{1,1}(K),\quad \Thi(\gamma)=1,\\
\Rop(K)\le\len(\gamma)\le\lambda\Rop(K)
\end{gathered}
\right\}.
\]
Indeed, after normalization one has $\len(\gamma)=\Rop(\gamma)$, and the lower length bound follows from the definition of $\Rop(K)$.  The window is nonempty because ropelength minimizers exist.  The case $\lambda=1$ consists exactly of normalized ropelength minimizers, while larger $\lambda$ allow controlled relaxation.
\end{remark}

\begin{proposition}\label{prop:rop-window-basic}
Let $K$ be a knot type and $p\in(-1,\infty]$.
\begin{enumerate}
\item[(1)] For every $\lambda\ge1$, $\rho^{\operatorname{rop}}_{p,\lambda}(K)$ is well defined and satisfies
\[
\rho^{\operatorname{rop}}_{p,\lambda}(K)\ge\rho_p(K)\ge2.
\]
\item[(2)] The function $\lambda\mapsto\rho^{\operatorname{rop}}_{p,\lambda}(K)$ is non-increasing on $[1,\infty)$.
\item[(3)] If $-1<p<q\le\infty$, then for every $\lambda\ge1$,
\[
\rho^{\operatorname{rop}}_{p,\lambda}(K)
\ge
\rho^{\operatorname{rop}}_{q,\lambda}(K).
\]
\end{enumerate}
\end{proposition}

\begin{proof}
The admissible class is nonempty and all ratios are finite and positive.  It is a subclass of the unconstrained class, which proves the first inequality in (1); the second is Proposition~\ref{prop:basic-properties}.  If $\lambda_1\le\lambda_2$, the admissible class for $\lambda_1$ is contained in that for $\lambda_2$, proving (2).  Finally, the admissible class does not depend on $p$, so (3) follows from Proposition~\ref{prop:p-monotonicity} by taking infima.
\end{proof}

\subsection{Existence of minimizers}

\begin{theorem}[Existence of ropelength-filtered minimizers]\label{thm:rop-window-existence}
Let $K$ be a knot type, let $p\in(-1,\infty]$, and let $\lambda\ge1$.  Then the infimum defining $\rho^{\operatorname{rop}}_{p,\lambda}(K)$ is attained.
\end{theorem}

\begin{proof}
By scale invariance, normalize a minimizing sequence $\{\gamma_j\}$ so that
\[
\Thi(\gamma_j)=1,
\qquad
\Rop(K)\le
\len(\gamma_j)\le\lambda\Rop(K).
\]
Translate each curve so that a chosen point lies at the origin.  Since the lengths are uniformly bounded, all images lie in a fixed compact ball.  Parametrize each $\gamma_j$ by constant speed on $S^1$.  The thickness bound gives a uniform curvature bound and a uniform embedded-tube radius.  By the standard compactness theorem for thick curves with bounded length, after passing to a subsequence the curves converge in $C^1$ to an embedded $C^{1,1}$ curve $\gamma$ with
\[
\Thi(\gamma)\ge1,
\qquad
\len(\gamma)\le\lambda\Rop(K),
\]
  and the knot type is preserved; see \cite{GMSvdM,CKS,Rawdon}.  More explicitly, the closedness conclusion in the cited thick-curve compactness theorem preserves $\Thi\ge1$, while sufficiently close curves lie in a fixed normal sub-tube of the limit and are ambient isotopic to it by nearest-point projection.  The uniform bound $\len(\gamma_j)\ge2\pi$ from Fenchel's theorem prevents collapse to a point.  Thus $\gamma\in\mathcal C^{1,1}(K)$ and
\[
\Rop(\gamma)
=
\frac{\len(\gamma)}{\Thi(\gamma)}
\le
\lambda\Rop(K).
\]

By Lemma~\ref{lem:Dp-continuity}, $\sD_p(\gamma_j)\to\sD_p(\gamma)$ for every $p\in(-1,\infty]$; for $p\le0$ the required uniform chord-arc estimate \eqref{eq:uniform-chord-arc} holds because $\Thi(\gamma_j)\ge1$ implies $|\kappa_{\gamma_j}|\le1$ almost everywhere.  Indeed, Schur comparison gives
\[
|\gamma_j(s)-\gamma_j(t)|
\ge
2\sin\!\left(\frac{d_{\gamma_j}(s,t)}2\right)
\]
whenever $d_{\gamma_j}(s,t)\le1$, and hence a uniform local chord-arc estimate.  The length is continuous under $C^1$ convergence.  Therefore
\[
\frac{\len(\gamma_j)}{\sD_p(\gamma_j)}
\longrightarrow
\frac{\len(\gamma)}{\sD_p(\gamma)}
\]
after passing to the minimizing subsequence, and $\gamma$ realizes the infimum.
\end{proof}

\begin{remark}
The thickness condition is precisely what fails in the local-knotting degeneration of the unconstrained problem.  A localized knot inserted in a ball of radius $r$ has thickness at most of order $r$ after smoothing.  Hence such a construction cannot remain in the normalized class $\Thi\ge1$ as $r\to0$.
\end{remark}

\subsection{Unknot detection and exhaustion}

\begin{lemma}[Restriction to positive-thickness representatives]\label{lem:smooth-restriction}
For every knot type $K$ and every $p\in(-1,\infty]$,
\[
\rho_p(K)
=
\inf_{\substack{\gamma\in\Hg(K)\\ \gamma\text{ smooth}}}
\frac{\len(\gamma)}{\sD_p(\gamma)}.
\]
Every representative in the infimum on the right has positive thickness and finite ropelength.
\end{lemma}

\begin{proof}
The right-hand infimum is at least $\rho_p(K)$.  Conversely, start with any $\gamma\in\Hg(K)$.  Corollary~\ref{cor:appendix-polygonal-full} gives polygonal representatives whose lengths and $\sD_p$ values converge to those of $\gamma$.  Lemma~\ref{lem:polygonal-smoothing-ratio} then smooths each polygon with an arbitrarily small further change in the same two quantities.  A diagonal choice produces smooth representatives of $K$ whose ratios converge to that of $\gamma$.  Taking the infimum over $\gamma$ proves equality.  A smooth embedded compact curve has a positive-radius normal tube by the tubular-neighborhood theorem, hence positive thickness.
\end{proof}

\begin{theorem}[Unknot detection by ropelength windows]\label{thm:rop-window-detects-unknot}
Let $p\in(-1,2]$ and $\lambda\ge1$.  Then
\[
\rho^{\operatorname{rop}}_{p,\lambda}(K)\ge c_p,
\]
with equality if and only if $K=U$.
\end{theorem}

\begin{proof}
Proposition~\ref{prop:rop-window-basic} and Theorem~\ref{thm:mean-chord-degeneration} give
\[
\rho^{\operatorname{rop}}_{p,\lambda}(K)
\ge
\rho_p(K)
=c_p.
\]
For $K=U$, a unit circle has ropelength $2\pi=\Rop(U)$ by Lemma~\ref{lem:unknot-ropelength}, so it is admissible for every $\lambda\ge1$ and has density $c_p$.  Thus equality holds for the unknot.

Conversely, suppose equality holds for $K$.  By Theorem~\ref{thm:rop-window-existence}, there is an admissible minimizer $\gamma$ with
\[
\frac{\len(\gamma)}{\sD_p(\gamma)}=c_p.
\]
The strict equality statement in Theorem~\ref{thm:mean-chord} implies that $\gamma$ is a round circle.  Hence its knot type is $U$.
\end{proof}

\begin{corollary}\label{cor:window-endpoints}
At $\lambda=1$, the admissible unknot representatives normalized by $\Thi=1$ are precisely the unit circles.  Moreover, for every knot type $K$ and every $\lambda\ge1$,
\[
\rho^{\operatorname{rop}}_{\infty,\lambda}(K)>2.
\]
\end{corollary}

\begin{proof}
The first assertion is Lemma~\ref{lem:unknot-ropelength}.  For the second, Theorem~\ref{thm:rop-window-existence} supplies a minimizing embedded curve.  Equality in the length--diameter inequality would force that curve to be a doubled segment, contradicting embeddedness; see the final paragraph of the proof of Theorem~\ref{thm:infty-degeneration}.
\end{proof}

\begin{theorem}[Exhaustion by ropelength windows]\label{thm:lambda-limit}
For every knot type $K$ and every $p\in(-1,\infty]$,
\[
\lim_{\lambda\to\infty}
\rho^{\operatorname{rop}}_{p,\lambda}(K)
=
\rho_p(K).
\]
\end{theorem}

\begin{proof}
By Proposition~\ref{prop:rop-window-basic}, the left-hand side exists and is at least $\rho_p(K)$.  Fix $\eps>0$.  Lemma~\ref{lem:smooth-restriction} gives a smooth representative $\gamma$ of $K$ such that
\[
\frac{\len(\gamma)}{\sD_p(\gamma)}
<
\rho_p(K)+\eps.
\]
Its ropelength is finite.  Therefore, for every
\[
\lambda
\ge
\frac{\Rop(\gamma)}{\Rop(K)},
\]
the curve $\gamma$ is admissible in Definition~\ref{def:rop-filtered-density}, and hence
\[
\rho^{\operatorname{rop}}_{p,\lambda}(K)
\le
\frac{\len(\gamma)}{\sD_p(\gamma)}
<
\rho_p(K)+\eps.
\]
Letting first $\lambda\to\infty$ and then $\eps\to0$ proves the result.
\end{proof}

\begin{theorem}[Continuity of the filtered spectrum]\label{thm:window-p-continuity}
For fixed $K$ and $\lambda\ge1$, the function
\[
p\longmapsto\rho^{\operatorname{rop}}_{p,\lambda}(K)
\]
is continuous on $(-1,\infty]$.
\end{theorem}

\begin{proof}
By scale invariance, the infimum may be taken over curves of thickness at least one.  Translate a chosen point to the origin, use constant-speed parametrizations, and set
\[
\mathcal A_\lambda(K):=
\left\{
\gamma\in\mathcal C^{1,1}(K):
\Thi(\gamma)\ge1,
\ \len(\gamma)\le\lambda\Rop(K)
\right\}.
\]
This class is compact in the $C^1$ topology by the same thick-curve compactness theorem used in Theorem~\ref{thm:rop-window-existence}.  The bound $\Thi\ge1$ supplies the uniform circular chord-arc estimate displayed in that proof, while the length cap and the choice of a base point give a common compact range.

On every compact exponent interval $[a,b]\subset(-1,\infty)$, the map
\[
(p,\gamma)
\longmapsto
\frac{\len(\gamma)}{\sD_p(\gamma)}
\]
is jointly continuous on $[a,b]\times\mathcal A_\lambda(K)$.  For fixed positive exponent, this is Lemma~\ref{lem:Dp-continuity}(a).  If the interval meets the negative range, the uniform chord-arc estimate dominates all kernels by a fixed integrable power with exponent greater than $-1$, as in Lemma~\ref{lem:uniform-integrability}.  At $p=0$, the same lemma gives the corresponding logarithmic domination by $1+|\log d_{S^1}|$, and Proposition~\ref{prop:pzero-limit} supplies the matching limit in the exponent.  These estimates are uniform over $\mathcal A_\lambda(K)$.

Joint continuity on a compact product is uniform.  Therefore
\[
\left|
\min_{\gamma\in\mathcal A_\lambda(K)}
\frac{\len(\gamma)}{\sD_p(\gamma)}
-
\min_{\gamma\in\mathcal A_\lambda(K)}
\frac{\len(\gamma)}{\sD_q(\gamma)}
\right|
\le
\sup_{\gamma\in\mathcal A_\lambda(K)}
\left|
\frac{\len(\gamma)}{\sD_p(\gamma)}
-
\frac{\len(\gamma)}{\sD_q(\gamma)}
\right|
\longrightarrow0
\]
as $q\to p$.  The minima equal the filtered densities, proving continuity.

It remains to treat $p=\infty$.  For finite $p>0$, set
\[
F_p(\gamma):=\frac{\len(\gamma)}{\sD_p(\gamma)},
\qquad
F_\infty(\gamma):=\frac{\len(\gamma)}{\diam(\gamma)}.
\]
Proposition~\ref{prop:p-monotonicity} and Lemma~\ref{lem:Lp-to-diameter} show that
\[
F_n(\gamma)\downarrow F_\infty(\gamma)
\]
pointwise on $\mathcal A_\lambda(K)$ as the integers $n\to\infty$.  Every $F_n$ is continuous by the finite-exponent argument, while $F_\infty$ is continuous by Lemma~\ref{lem:Dp-continuity}(a).  Dini's theorem therefore gives uniform convergence on the compact set $\mathcal A_\lambda(K)$.  Hence
\[
\lim_{p\to\infty}\rho^{\operatorname{rop}}_{p,\lambda}(K)
=
\rho^{\operatorname{rop}}_{\infty,\lambda}(K);
\]
monotonicity in $p$ extends the conclusion from integer to arbitrary exponents.
\end{proof}

\begin{proposition}[Right-continuity in the window]\label{prop:lambda-right-continuity}
For fixed $K$ and $p\in(-1,\infty]$, the non-increasing function
\[
\lambda\longmapsto\rho^{\operatorname{rop}}_{p,\lambda}(K)
\]
is right-continuous on $[1,\infty)$.
\end{proposition}

\begin{proof}
Let $\lambda_j\downarrow\lambda$.  For each $j$, choose a normalized minimizer $\gamma_j$ in the window $\lambda_j$.  All $\gamma_j$ lie in the compact class $\mathcal A_{\lambda_1}(K)$.  After passage to a subsequence, they converge in $C^1$ to a curve $\gamma$ of the same knot type with
\[
\Thi(\gamma)\ge1,
\qquad
\len(\gamma)\le\lambda\Rop(K).
\]
Thus $\gamma\in\mathcal A_\lambda(K)$.  Lemma~\ref{lem:Dp-continuity}, applied as in the proof of Theorem~\ref{thm:rop-window-existence}, gives
\[
\lim_{j\to\infty}
\rho^{\operatorname{rop}}_{p,\lambda_j}(K)
=
\frac{\len(\gamma)}{\sD_p(\gamma)}
\ge
\rho^{\operatorname{rop}}_{p,\lambda}(K).
\]
The reverse inequality follows from monotonicity because $\lambda_j>\lambda$.  Therefore equality holds.
\end{proof}

\begin{figure}[H]
\centering
\begin{tikzpicture}[>=Stealth,font=\small]
\draw[->] (0,0)--(9.6,0) node[right]{$\lambda$};
\draw[->] (0,0)--(0,4.2);
\draw (0.8,0.08)--(0.8,-0.08) node[below]{$1$};
\draw (3.6,0.08)--(3.6,-0.08) node[below]{$\lambda_0$};
\draw[gray,densely dashed] (0,0.8)--(9.3,0.8);
\node[left] at (0,0.8) {$\rho_p(K)=\rho_p(U)$};
\draw[gray,densely dotted] (0,3.6)--(0.8,3.6);
\node[left] at (0,3.6) {$\rho^{\operatorname{rop}}_{p,1}(K)$};
\draw[gray,densely dotted] (0.8,0)--(0.8,3.6);
\draw[gray,densely dotted] (3.6,0)--(3.6,2.0);
\draw[thick] (0.8,3.6) .. controls (1.6,3.3) and (2.6,3.0) .. (3.6,2.9);
\draw[thick] (3.6,2.0) .. controls (5,1.3) and (7,0.95) .. (9.3,0.88);
\fill (0.8,3.6) circle (1.8pt);
\draw[thick,fill=white] (3.6,2.9) circle (1.8pt);
\fill (3.6,2.0) circle (1.8pt);
\end{tikzpicture}
\caption{Schematic profile of $\lambda\mapsto\rho^{\operatorname{rop}}_{p,\lambda}(K)$ for fixed $K$ and $p$: non-increasing and right-continuous (Propositions~\ref{prop:rop-window-basic} and~\ref{prop:lambda-right-continuity}), starting from the value on the ideal stratum at $\lambda=1$ and decreasing to the knot-type-independent limit $\rho_p(K)$ (Theorem~\ref{thm:lambda-limit}).  The jump at $\lambda_0$ illustrates a possibility that is not excluded by our results; see Question~\ref{q:lambda-left-continuity}.}
\label{fig:lambda-profile}
\end{figure}

\begin{proposition}[A ropelength growth bound at the diameter endpoint]\label{prop:diameter-window-upper-bound}
Let $\lambda\ge1$, set
\[
L:=\Rop(K),
\qquad
T:=\lambda L,
\]
and define
\[
c_{\mathrm J}:=
\sqrt{\frac83}\left(\frac34\right)^{1/3}
=1.48367\ldots.
\]
Then
\[
2<
\rho^{\operatorname{rop}}_{\infty,\lambda}(K)
\le
\frac{T}{c_{\mathrm J}T^{1/3}-2}
=O\!\left(T^{2/3}\right).
\]
\end{proposition}

\begin{proof}
Choose a minimizing curve for $\rho^{\operatorname{rop}}_{\infty,\lambda}(K)$ and normalize it to thickness one.  Write
\[
\ell:=\len(\gamma),
\qquad
D:=\diam(\gamma).
\]
Then $L\le\ell\le T$.  For $0<r<1$, the radius-$r$ normal tube is embedded and has volume $\pi r^2\ell$.  Its diameter is at most $D+2r$.  Jung's theorem in $\RR^3$ \cite{Jung} places it in a ball of radius at most
\[
\sqrt{\frac38}\,(D+2r).
\]
Comparison with the volume of this ball gives
\[
\pi r^2\ell
\le
\frac{4\pi}{3}
\left(\sqrt{\frac38}\,(D+2r)\right)^3.
\]
Letting $r\uparrow1$ yields
\[
D\ge c_{\mathrm J}\ell^{1/3}-2.
\]
The denominator is positive because $\ell\ge L\ge2\pi$.  Hence, with
\[
f(t):=\frac{t}{c_{\mathrm J}t^{1/3}-2},
\]
the chosen curve satisfies
\[
\rho^{\operatorname{rop}}_{\infty,\lambda}(K)
=\frac{\ell}{D}
\le f(\ell).
\]
The function $f$ is increasing for
\[
t>t_0:=\left(\frac{3}{c_{\mathrm J}}\right)^3
=8.27\ldots.
\]
If $K$ is nontrivial, the quadrisecant lower bound \cite{DDS} gives $L\ge15.66>t_0$, and therefore $f(\ell)\le f(T)$.  If $K=U$, the unit circle is admissible and gives
\[
\rho^{\operatorname{rop}}_{\infty,\lambda}(U)\le\pi.
\]
Moreover, $f$ has its minimum on $((2/c_{\mathrm J})^3,\infty)$ at $t_0$, and
\[
f(t_0)=t_0>\pi.
\]
Since $T\ge2\pi$, this again gives
\[
\rho^{\operatorname{rop}}_{\infty,\lambda}(U)<f(T).
\]
The strict lower bound is Corollary~\ref{cor:window-endpoints}, and the displayed asymptotic follows from the explicit formula for $f(T)$.
\end{proof}

\begin{question}\label{q:lambda-rate}
Can the rate at which $\rho^{\operatorname{rop}}_{p,\lambda}(K)$ approaches $\rho_p(K)$ be estimated in terms of $K$, $p$, or quantitative local-knotting constructions?
\end{question}

\subsection{Numerical evidence: granny versus square knot}\label{subsec:granny-square}

The exact results above separate the unknot from every nontrivial knot type in the mean-chord range, but they do not by themselves separate two nontrivial knot types.  A natural first test is the pair
\[
K_G:=3_1\#3_1,
\qquad
K_S:=3_1\#\overline{3_1},
\]
the granny and square knots.  These have the same prime factors up to mirroring, but the handedness of the summands is different.  Cantarella--LaPointe--Rawdon computed approximately tight configurations of both knots using Ridgerunner and reported polygonal ropelengths $57.05$ for $K_G$ and $57.09$ for $K_S$ \cite{CLR}.  Full-precision coordinate data for the approximately tight polygons are publicly available \cite{TightKnotData}; the two coordinate files used here each contain $457$ vertices.

Let $P_G$ and $P_S$ denote these published polygons.  Because they are polygonal numerical approximations, rather than proved global $C^{1,1}$ ropelength minimizers, we do \emph{not} identify the following sampled quantities with the exact invariant $\rho^{\operatorname{rop}}_{p,1}$.  For convenience write
\[
\widehat\rho^{\mathrm{tight}}_p(K_G)
:=\frac{\len(P_G)}{\sD_p(P_G)},
\qquad
\widehat\rho^{\mathrm{tight}}_p(K_S)
:=\frac{\len(P_S)}{\sD_p(P_S)}.
\]
The ratios are similarity-invariant, so the scale in which the coordinate files are stored is irrelevant.

For $p=2$ the polygonal computation can be performed directly from exact arclength moments.  If the vertices of a closed polygon $P$ are $v_i$, with
\[
d_i:=v_{i+1}-v_i,
\qquad
\ell_i:=|d_i|,
\qquad
L:=\sum_i\ell_i,
\]
then its arclength center of mass is
\[
c_P=\frac1L\sum_i\frac{\ell_i}{2}(v_i+v_{i+1}),
\]
and
\[
R_g(P)^2
=
\frac1L\sum_i\ell_i
\left(
|v_i|^2+v_i\cdot d_i+\frac{|d_i|^2}{3}
\right)
-|c_P|^2.
\]
Thus Proposition~\ref{prop:p=2} gives $\sD_2(P)=\sqrt2\,R_g(P)$.  For the published polygons we obtain
\[
\begin{aligned}
\len(P_G)&=28.5272386043,
& R_g(P_G)&=1.72199904189,\\
\len(P_S)&=28.5415727667,
& R_g(P_S)&=1.76546392117,
\end{aligned}
\]
which gives the first row of Table~\ref{tab:granny-square-density}; the granny value is about $2.47\%$ larger than the square value.

For finite $p>0$, the remaining double integrals were evaluated edge-pair by edge-pair, and the implementation was checked against the analytic $p=2$ identity, the diameter endpoint, and quadrature-order convergence.  Distinct edge pairs were integrated by Gauss--Legendre quadrature, while the diagonal edge contribution was inserted analytically:
\[
\int_{e_i}\int_{e_i}|x-y|^p\,ds_x\,ds_y
=
\frac{2\ell_i^{p+2}}{(p+1)(p+2)}.
\]
For the even integer exponents displayed below, an $8$-, $12$-, and $16$-point Gauss--Legendre rule on each parameter interval gave the displayed values unchanged to the digits shown.  At $p=\infty$, convexity of the distance on each product of two line segments implies that the polygonal diameter is realized by vertices.  The resulting diameters are
\[
\diam(P_G)=5.00826424854,
\qquad
\diam(P_S)=4.96561561006,
\]
so
\[
\widehat\rho^{\mathrm{tight}}_\infty(K_G)=5.69603303,
\qquad
\widehat\rho^{\mathrm{tight}}_\infty(K_S)=5.74784176.
\]
Thus the ordering at the diameter endpoint is the reverse of the ordering at $p=2$; see Table~\ref{tab:granny-square-density} and Figure~\ref{fig:granny-square-crossing}.

\begin{table}[H]
\centering
\begin{tabular}{c|ccc}
$p$ & $\widehat\rho^{\mathrm{tight}}_p(K_G)$
& $\widehat\rho^{\mathrm{tight}}_p(K_S)$
& $\Delta_p:=\widehat\rho^{\mathrm{tight}}_p(K_G)-\widehat\rho^{\mathrm{tight}}_p(K_S)$\\
\hline
$2$ & $11.71417833$ & $11.43152200$ & $\phantom{-}0.28265634$\\
$4$ & $10.25065332$ & $9.94972478$ & $\phantom{-}0.30092854$\\
$8$ & $8.77141549$ & $8.53179943$ & $\phantom{-}0.23961607$\\
$16$ & $7.58465863$ & $7.46419821$ & $\phantom{-}0.12046042$\\
$32$ & $6.78904776$ & $6.76656348$ & $\phantom{-}0.02248428$\\
$39.860593$ & $6.60593611$ & $6.60593611$ & $\approx0$\\
$40$ & $6.60325448$ & $6.60357822$ & $-0.00032375$\\
$\infty$ & $5.69603303$ & $5.74784176$ & $-0.05180873$
\end{tabular}
\caption{Sampled $p$-densities of the published approximately tight granny and square polygons.  These are values on the numerical tight configurations, not asserted exact values of $\rho^{\operatorname{rop}}_{p,1}$.}
\label{tab:granny-square-density}
\end{table}

\begin{figure}[H]
\centering
\begin{tikzpicture}
\begin{axis}[
width=0.82\textwidth,
height=0.46\textwidth,
xlabel={$p$},
ylabel={$\Delta_p$},
xmin=2,xmax=40.5,
ymin=-0.025,ymax=0.325,
grid=major,
]
\addplot[thick,mark=*] coordinates {
(2,0.28265634)
(4,0.30092854)
(8,0.23961607)
(16,0.12046042)
(32,0.02248428)
(39.860593,0)
(40,-0.00032375)
};
\addplot[gray,densely dashed] coordinates {(2,0) (40.5,0)};
\addplot[only marks,mark=*,mark size=1.8pt] coordinates {(39.860593,0)};
\node[anchor=south east,font=\scriptsize] at (axis cs:39.5,0.008) {$p_{\mathrm{GS}}\approx39.8606$};
\end{axis}
\end{tikzpicture}
\caption{Difference $\Delta_p$ between the sampled tight-shape densities of the granny and square knots.  The sign change near $p_{\mathrm{GS}}=39.860593\ldots$ shows that the geometric ordering of these two approximately tight shapes depends on the exponent.}
\label{fig:granny-square-crossing}
\end{figure}

Since $p\mapsto\len(P)/\sD_p(P)$ is continuous for each fixed polygon on $(-1,\infty]$, the opposite signs at $p=2$ and $p=\infty$ imply at least one crossing of the two sampled profiles.  Numerically, the crossing shown in Table~\ref{tab:granny-square-density} occurs at
\[
p_{\mathrm{GS}}=39.860593\ldots,
\qquad
\widehat\rho^{\mathrm{tight}}_{p_{\mathrm{GS}}}(K_G)
=
\widehat\rho^{\mathrm{tight}}_{p_{\mathrm{GS}}}(K_S)
\approx6.60593611.
\]
This calculation gives concrete evidence that a ropelength-filtered density can retain information about the handedness of summands after the unconstrained density has completely degenerated.  However, proving
\[
\rho^{\operatorname{rop}}_{p,1}(K_G)
\ne
\rho^{\operatorname{rop}}_{p,1}(K_S)
\]
for any $p$ would require control of the true ropelength-minimizing strata and of the density minimization within those strata.  The numerical experiment should therefore be read as a reproducible test case and motivation for Question~\ref{q:nontrivial-separation}, not as a theorem about the exact filtered invariants.

\section{Polygonal approximation of the unconstrained densities}\label{sec:polygonal-approximation}

\begin{theorem}[Polygonal approximation theorem]\label{thm:polygonal-approximation}
Let $K$ be a knot type and let $p\in(-1,\infty]$.  Then
\[
\lim_{n\to\infty}\rho_{p,n}(K)=\rho_p(K).
\]
\end{theorem}

\begin{proof}
Proposition~\ref{prop:basic-properties}(4) shows that the sequence $\{\rho_{p,n}(K)\}_{n\ge e(K)}$ is non-increasing.  It is bounded below, so the limit exists.

Fix $\eps>0$.  Choose $\gamma\in\Hg(K)$ such that
\[
\frac{\len(\gamma)}{\sD_p(\gamma)}<\rho_p(K)+\eps.
\]
By Corollary~\ref{cor:appendix-polygonal-full}, there exists a sequence of embedded polygonal knots $P_m$ representing $K$ such that
\[
\len(P_m)\to\len(\gamma),
\qquad
\sD_p(P_m)\to\sD_p(\gamma).
\]
Thus, for all sufficiently large $m$,
\[
\frac{\len(P_m)}{\sD_p(P_m)}<\rho_p(K)+2\eps.
\]
Let $n_m$ be the number of edges of $P_m$.  By subdividing edges if necessary, we may assume $n_m\to\infty$ without changing length or $\sD_p$.  Hence
\[
\rho_{p,n_m}(K)\le\frac{\len(P_m)}{\sD_p(P_m)}<\rho_p(K)+2\eps.
\]
Since $\rho_{p,n}(K)$ is non-increasing in $n$, it follows that
\[
\limsup_{n\to\infty}\rho_{p,n}(K)\le\rho_p(K)+2\eps.
\]
Letting $\eps\to0$ gives the upper inequality.  The reverse inequality follows from Proposition~\ref{prop:basic-properties}(3), which says $\rho_p(K)\le\rho_{p,n}(K)$ for every $n$.  Therefore the limit equals $\rho_p(K)$.
\end{proof}

For the universal baseline below, let $\mathfrak P_n$ denote the class of nonconstant closed polygonal curves having at most $n$ edges.  Unlike $\Poly_n(K)$ from Definition~\ref{def:polygonal-knots}, this auxiliary class need not be embedded or represent a knot type; repeated edges and zero edges are allowed, with zero edges recording curves having fewer than $n$ edges.

\begin{proposition}[Universal strict fixed-edge gap]\label{prop:universal-polygon-gap}
Let $p\in(0,2]$ and $n\ge3$.  There is a universal constant
\[
m_{p,n}:=
\min_{P\in\mathfrak P_n}
\frac{\len(P)}{\sD_p(P)}
\]
such that
\[
m_{p,n}>c_p.
\]
For every knot type $K$ with $\Poly_n(K)\ne\varnothing$,
\[
\rho_{p,n}(K)\ge m_{p,n}>c_p.
\]
Moreover, $m_{p,n}$ is non-increasing in $n$ and
\[
\lim_{n\to\infty}m_{p,n}=c_p.
\]
\end{proposition}

\begin{proof}
By scale invariance, normalize to length one and fix the initial vertex at the origin.  A possibly degenerate closed polygonal curve with at most $n$ edges is represented by edge vectors in the compact set
\[
X_n:=
\left\{
(v_1,\ldots,v_n)\in(\RR^3)^n:
\sum_{i=1}^n v_i=0,\quad
\sum_{i=1}^n|v_i|=1
\right\};
\]
zero vectors record unused edges.  If $x_i=\sum_{k<i}v_k$, then
\[
\sD_p(P)^p
=
\sum_{i,j=1}^n |v_i||v_j|
\int_0^1\int_0^1
\left|x_i+t v_i-x_j-s v_j\right|^p\,dt\,ds.
\]
For $p>0$, this is a continuous function of the edge vectors, including when some $v_i$ vanish.  It is positive on $X_n$ because the total length is one.  Hence $\len(P)/\sD_p(P)=1/\sD_p(P)$ has a minimum on $X_n$.

Theorem~\ref{thm:mean-chord} gives $\len(P)/\sD_p(P)\ge c_p$.  Equality would force the polygonal curve to have the image of a round circle, which no finite polygonal curve can do.  Since the minimum is attained, it follows that $m_{p,n}>c_p$.  Every member of $\Poly_n(K)$, after length normalization, lies in $X_n$, which proves the knot-independent bound.

Adding a zero edge or subdividing an edge shows $m_{p,n+1}\le m_{p,n}$.  Regular planar $n$-gons converge in arc-length parametrization and length to a round circle, so Lemma~\ref{lem:Dp-continuity}(a) gives
\[
\limsup_{n\to\infty}m_{p,n}\le c_p.
\]
The reverse inequality holds for every $n$, completing the proof.
\end{proof}

\begin{remark}
Because $\mathfrak P_n$ contains degenerate curves, the universal baseline can be determined by a lower-edge degeneration when $n$ is small.  In particular, $\mathfrak P_2\subset\mathfrak P_n$ consists of doubled segments, so, with the same definition of $m_{p,2}$,
\[
m_{p,n}\le m_{p,2}=d_p.
\]
This already rules out the regular triangle at $p=2$.  Indeed, a regular triangle $R_3$ of side length $a$ has radius of gyration $a/\sqrt6$, so Proposition~\ref{prop:p=2} gives
\[
\frac{\len(R_3)}{\sD_2(R_3)}=3\sqrt3
>2\sqrt6=d_2.
\]
Thus $m_{2,3}$ cannot be realized by $R_3$.

Although the limiting value is the degenerate unconstrained constant, the finite sequence
\[
\{\rho_{p,n}(K)\}_{n\ge e(K)}
\]
may retain knot information above the universal baseline $m_{p,n}$: a localized knot insertion consumes edges, especially near the stick number $e(K)$.  L\"uk\H{o}'s work \cite{Luko} includes related discrete mean-chord extremal results for equilateral polygons, while Proposition~\ref{prop:universal-polygon-gap} allows arbitrary edge lengths.

A polygonal analogue of $\rho^{\operatorname{rop}}_{p,\lambda}$ would instead involve polygonal thickness or another discrete tube-radius model.  Establishing convergence of that constrained polygonal theory is a separate problem, related to but distinct from the fixed-edge profile above.

The restriction $p>0$ in Proposition~\ref{prop:universal-polygon-gap} is used in the continuity argument on the compact degenerate class $X_n$.  Extending the proposition to $-1<p<0$ would require a separate uniform-integrability estimate near overlapping edges; the chord-arc estimates in the appendix do not apply directly to these degenerations.
\end{remark}

For $n\ge3$, let $R_n$ denote a planar regular $n$-gon and set
\[
r_n(p):=\frac{\len(R_n)}{\sD_p(R_n)}.
\]

\begin{question}\label{q:polygonal-universal-minimizer}
For $p\in(0,2]$ and $n\ge3$, does
\[
m_{p,n}\overset{?}{=}\min\{d_p,r_n(p)\}
\]
hold?  In particular, whenever $r_n(p)<d_p$, is $m_{p,n}$ realized by the planar regular $n$-gon $R_n$?  More generally, does the excess
\[
\rho_{p,n}(K)-m_{p,n}
\]
distinguish knot types at some fixed edge levels?
\end{question}

\section{Questions and further directions}\label{sec:questions}

We close with a short list of directions where additional estimates would be needed before the framework yields stronger geometric and knot-theoretic information.

\subsection{The sharp finite high-exponent extremal problem}

By Theorem~\ref{thm:finite-high-degeneration}, the determination of the common value of $\rho_p$ for finite $p>2$ is equivalent to a purely planar convex extremal problem.

\begin{question}[Sharp planar convex extremizers]\label{q:sharp-planar-extremizer}
For each finite $p>2$, which convex plane curves maximize
\[
\frac{\sD_p(\gamma)}{\len(\gamma)}?
\]
Is the maximizer unique up to similarities and reparametrization, and what regularity does it have?  How does it depend on $p$?  In particular, determine the critical behavior at which the round circle loses global optimality and describe the convergence of maximizers toward the doubly covered segment as $p\to\infty$.
\end{question}

\subsection{Ropelength-filtered spectra}

For fixed $K$ and $\lambda$, the function $p\mapsto \rho^{\operatorname{rop}}_{p,\lambda}(K)$ is the constrained analogue of the unconstrained $p$-spectrum.  Proposition~\ref{prop:rop-window-basic} and Theorem~\ref{thm:window-p-continuity} show that it is non-increasing and continuous on $(-1,\infty]$.  In the sharp mean-chord range, Theorem~\ref{thm:rop-window-detects-unknot} gives the much stronger bound
\[
\rho^{\operatorname{rop}}_{p,\lambda}(K)\ge c_p
\ge c_2=\sqrt2\pi
\qquad (-1<p\le2),
\]
and equality in the first inequality detects the unknot.  Thus the elementary lower bound by $2$ is relevant mainly near the diameter endpoint; even there, Corollary~\ref{cor:window-endpoints} gives strict inequality for every $\lambda\ge1$.  At $\lambda=1$, the problem restricts to normalized ropelength minimizers, so $\rho^{\operatorname{rop}}_{p,1}(K)$ records geometric features of the ideal stratum rather than a localized-knotting collapse.

\begin{question}\label{q:rop-spectrum}
Beyond the sharp bound for $p\le2$, can one obtain lower bounds in terms of geometric or topological quantities such as ropelength, trunk, representativity, bridge number, or bridge distance?
\end{question}

Proposition~\ref{prop:diameter-window-upper-bound} sets the possible growth scale for such estimates at $p=\infty$, namely $O(T^{2/3})$ with $T=\lambda\Rop(K)$.  In particular, at any fixed window no lower bound proportional to $\Rop(K)$ can hold in general.  Known lower bounds for the ropelength of nontrivial knot types, including the quadrisecant bounds of Denne--Diao--Sullivan \cite{DDS}, provide a natural starting point, but do not by themselves control the diameter of a constrained minimizing representative.

Proposition~\ref{prop:lambda-right-continuity} proves right-continuity in the window parameter.  Left-continuity is subtler because a new minimizing branch may enter exactly when its ropelength reaches the boundary of the window.

\begin{question}\label{q:lambda-left-continuity}
Is the function
\[
\lambda\longmapsto\rho^{\operatorname{rop}}_{p,\lambda}(K)
\]
left-continuous?  If jumps occur, which geometric transitions among constrained ideal representatives do their locations record?
\end{question}

The rigorous results separate the unknot from every nontrivial knot type, but they do not yet prove separation of two prescribed nontrivial knot types.  The granny--square computation in Subsection~\ref{subsec:granny-square} gives numerical evidence that such separation may occur already at the ideal-stratum boundary $\lambda=1$.

\begin{question}\label{q:nontrivial-separation}
Do some parameters $p$ and $\lambda$ make these filtered densities distinguish prescribed nontrivial knot types?  More strongly, does the full two-parameter profile
\[
\left\{\rho^{\operatorname{rop}}_{p,\lambda}(K)\right\}_{p,\lambda}
\]
determine any standard knot-theoretic information beyond unknot detection?
\end{question}

\subsection{Discrete constrained theories}

\begin{question}\label{q:polygonal-filtered}
Can one define polygonal ropelength-filtered densities using polygonal thickness so that the resulting discrete quantities converge to $\rho^{\operatorname{rop}}_{p,\lambda}(K)$ as the number of edges tends to infinity?
\end{question}

\subsection{Regularized inverse-power extensions}

For $p\le -1$, the unregularized pairwise-distance kernel is not integrable near the diagonal.  A continuation beyond this threshold should therefore use regularized inverse-power energies, in the spirit of O'Hara energies, M\"obius energy, and Menger-curvature knot energies \cite{Ohara,Ohara2,Ohara3,OharaBook,BrysonFHW,FHW,SvdM,SSvdM}.

\begin{question}\label{q:strongly-repulsive}
Can one construct a unified regularized theory extending the present family beyond $p=-1$, so that the positive side measures spatial spread while the strongly negative side captures self-repulsive relaxation?
\end{question}


\appendix
\renewcommand{\thetheorem}{A.\arabic{theorem}}

\section{Technical lemmas}\label{app:technical}

This appendix collects the technical facts used in the continuity and approximation arguments.  The only point requiring some care is that all local estimates must be made on the parameter circle, not merely on the square $[0,1]^2$ with the ordinary distance.  We therefore write
\[
d_{S^1}(u,v)=\min\{|u-v|,1-|u-v|\},\qquad u,v\in[0,1]/0\sim1,
\]
and use the Lebesgue measure $du$ on this model of $S^1$, normalized so that $\int_{S^1}du=1$.  For an arc-length parameter circle of total length $L$, write
\[
d_L(s,t)=\min_{k\in\mathbb Z}|s-t+kL|.
\]

\begin{lemma}\label{lem:hausdorff-diameter}
Let $\{X_m\}_{m\ge 1}$ be a sequence of nonempty compact subsets of $\RR^3$ converging to a nonempty compact set $X$ in the Hausdorff metric.
Then
\[
\diam(X_m)\to\diam(X).
\]
\end{lemma}

\begin{proof}
This is standard; see, for example, \cite{BBI01}.
\end{proof}

\begin{lemma}\label{lem:local-chord-arc}
Let $\gamma\colon S^1\to\RR^3$ be a tame $C^1$ embedded closed curve, and let
\[
\widetilde\gamma\colon\RR/L\mathbb Z\to\RR^3
\]
be an arc-length parametrization.  Then there exist constants $c_0,\delta_0>0$ such that
\[
|\widetilde\gamma(s)-\widetilde\gamma(t)|\ge c_0\,d_L(s,t)
\]
whenever $d_L(s,t)\le\delta_0$.  Moreover,
\[
|\widetilde\gamma(s)-\widetilde\gamma(t)|\le d_L(s,t)
\]
for all $s,t$.

Sufficiently fine inscribed polygonal approximations of $\gamma$ inside shrinking tubular neighborhoods may be chosen so that, after rescaled arc-length parametrization on $S^1$, the same type of estimate
\[
|\eta_m(u)-\eta_m(v)|\ge c_1\,d_{S^1}(u,v)
\]
holds for $d_{S^1}(u,v)\le\delta_1$, with constants $c_1,\delta_1>0$ independent of $m$.
\end{lemma}

\begin{proof}
The upper bound follows because $d_L(s,t)$ is the shorter arc length between the two parameters.
For the lower bound, fix $s_0\in\RR/L\mathbb Z$.  Since $\widetilde\gamma$ is $C^1$ and parametrized by arc length,
\[
\lim_{s,t\to s_0}\frac{|\widetilde\gamma(s)-\widetilde\gamma(t)|}{d_L(s,t)}=1.
\]
Thus a lower chord-arc estimate holds in a small neighborhood of each point of the parameter circle.  Compactness of the parameter circle gives uniform constants $c_0$ and $\delta_0$.

For the polygonal approximations, choose inscribed polygons whose mesh in arc length tends to zero and whose images lie in tubular neighborhoods shrinking to $\gamma$.  The chord directions of such approximations converge uniformly to the tangent direction of $\gamma$ on short arcs.  Hence, after possibly decreasing the constant and the radius of locality, the same lower chord-arc estimate holds uniformly for all sufficiently large $m$; this is the standard polygonal approximation theorem for $C^1$ knots, in the form used for length and energy convergence, see \cite{Hirsch76,RawdonSimon06}.
\end{proof}

\begin{lemma}\label{lem:uniform-separation-away}
Let $\eta_m,\eta\colon S^1\to\RR^3$ be continuous maps such that $\eta_m\to\eta$ uniformly and $\eta$ is embedded.  For every $\delta>0$ there are $m_0$ and $\kappa>0$ such that
\[
|\eta_m(u)-\eta_m(v)|\ge\kappa
\]
for all $m\ge m_0$ whenever $d_{S^1}(u,v)\ge\delta$.
\end{lemma}

\begin{proof}
The compact set
\[
\{(u,v)\in S^1\times S^1: d_{S^1}(u,v)\ge\delta\}
\]
is disjoint from the diagonal.  Since $\eta$ is embedded, the continuous function $|\eta(u)-\eta(v)|$ has a positive minimum on this set.  Uniform convergence then gives the same lower bound, up to a factor of $1/2$, for $\eta_m$ for all sufficiently large $m$.
\end{proof}

\begin{lemma}[Uniform integrability]\label{lem:uniform-integrability}
Let $\eta_m,\eta\colon S^1\to\RR^3$ be rescaled arc-length parametrizations of rectifiable closed curves with lengths $L_m\to L>0$.  Assume that $\eta_m\to\eta$ uniformly, that $\eta$ is embedded, and that there exist constants $c,\delta>0$ such that
\begin{equation}\label{eq:uniform-chord-arc}
|\eta_m(u)-\eta_m(v)|\ge c\,d_{S^1}(u,v)
\qquad\text{whenever }d_{S^1}(u,v)\le\delta,
\end{equation}
for all sufficiently large $m$.  Then the kernels $\log|\eta_m(u)-\eta_m(v)|$ and, for each $p\in(-1,0)$, the kernels $|\eta_m(u)-\eta_m(v)|^p$ are uniformly integrable on $S^1\times S^1$.
\end{lemma}

\begin{proof}
Since $\eta_m$ has constant speed $L_m$, one has $|\eta_m(u)-\eta_m(v)|\le L_m\,d_{S^1}(u,v)$ for all $u,v$.  On the region $d_{S^1}(u,v)\ge\delta$, Lemma~\ref{lem:uniform-separation-away} gives a uniform positive lower bound for $|\eta_m(u)-\eta_m(v)|$, so both families of kernels are uniformly bounded there.  On the region $d_{S^1}(u,v)<\delta$, the two-sided estimate
\[
c\,d_{S^1}(u,v)\le|\eta_m(u)-\eta_m(v)|\le L_m\,d_{S^1}(u,v)
\]
gives
\[
|\eta_m(u)-\eta_m(v)|^p\le c^p d_{S^1}(u,v)^p,
\qquad
\bigl|\log|\eta_m(u)-\eta_m(v)|\bigr|\le C'+|\log d_{S^1}(u,v)|,
\]
with $C'$ independent of $m$.  The functions $d_{S^1}(u,v)^p$, for $p>-1$, and $|\log d_{S^1}(u,v)|$ are integrable on $S^1\times S^1$.  This proves uniform integrability.
\end{proof}

\begin{lemma}[Continuity of the spread functional]\label{lem:Dp-continuity}
Let $\gamma_m$ and $\gamma$ be rectifiable parametrized closed curves with lengths $L_m\to L>0$, and let $\eta_m,\eta\colon S^1\to\RR^3$ be their rescaled arc-length parametrizations.  Interpret $\sD_p$ by the same normalized parameter integral as in Definition~\ref{def:spread}.  Assume that $\eta_m\to\eta$ uniformly.
\begin{enumerate}
\item[(a)] For every $p\in(0,\infty]$, one has $\sD_p(\gamma_m)\to\sD_p(\gamma)$.  Here the curves need not be embedded.
\item[(b)] If $\eta$ is embedded and the uniform chord-arc estimate \eqref{eq:uniform-chord-arc} holds, then $\sD_p(\gamma_m)\to\sD_p(\gamma)$ also for every $p\in(-1,0]$.
\end{enumerate}
\end{lemma}

\begin{proof}
After rescaling to $S^1$, one has
\[
\sD_p(\gamma_m)^p=
\int_{S^1}\int_{S^1}|\eta_m(u)-\eta_m(v)|^p\,du\,dv
\quad(p\ne0,\infty),
\qquad
\log\sD_0(\gamma_m)=
\int_{S^1}\int_{S^1}\log|\eta_m(u)-\eta_m(v)|\,du\,dv.
\]

(a) For finite $p>0$, uniform convergence of $\eta_m$ implies uniform convergence of the kernels, because $x\mapsto x^p$ is continuous on a compact interval $[0,M]$ containing all distances.  For $p=\infty$, uniform convergence implies Hausdorff convergence of the images, and Lemma~\ref{lem:hausdorff-diameter} applies.

(b) Since $\eta$ is embedded, the kernels converge pointwise off the diagonal to the corresponding kernels for $\eta$.  By Lemma~\ref{lem:uniform-integrability} they are uniformly integrable, so Vitali's theorem gives convergence of the double integrals.  Continuity of $t\mapsto t^{1/p}$ on $(0,\infty)$, respectively of the exponential, gives the convergence of $\sD_p$.
\end{proof}
\begin{lemma}\label{lem:appendix-polygonal-approx}
Let $\gamma\in\Hg(K)$ be a tame $C^1$ representative.  Then there exists a sequence of embedded polygonal knots $\{P_m\}_{m\ge1}$ representing the same knot type $K$ such that:
\begin{enumerate}
\item[(i)] $P_m$ is ambient isotopic to $\gamma$ for every $m$;
\item[(ii)] $\im(P_m)\to\im(\gamma)$ in the Hausdorff metric;
\item[(iii)] $\len(P_m)\to\len(\gamma)$;
\item[(iv)] after rescaled arc-length parametrization on $S^1$, the maps $P_m$ converge uniformly to the corresponding parametrization of $\gamma$;
\item[(v)] the parametrizations satisfy a uniform circular local chord-arc estimate as in Lemma~\ref{lem:local-chord-arc}.
\end{enumerate}
\end{lemma}

\begin{proof}
Take sufficiently fine inscribed polygonal approximations of $\gamma$ inside a tubular neighborhood whose radius tends to zero.  For fine enough mesh, these polygons are ambient isotopic to $\gamma$, converge uniformly and in the Hausdorff metric to $\gamma$, and have lengths converging to $\len(\gamma)$.  The uniform circular chord-arc estimate is the final assertion of Lemma~\ref{lem:local-chord-arc}.  This is the standard polygonal approximation theorem for tame $C^1$ knots; see \cite{Hirsch76,RawdonSimon06}.
\end{proof}

\begin{corollary}\label{cor:appendix-polygonal-full}
Let $K$ be a knot type, let $p\in(-1,\infty]$, and let $\gamma\in\Hg(K)$ be a tame $C^1$ representative.
Then there exists a sequence of embedded polygonal knots $\{P_m\}$ representing $K$ such that
\[
\len(P_m)\to\len(\gamma)
\qquad\text{and}\qquad
\sD_p(P_m)\to\sD_p(\gamma).
\]
\end{corollary}

\begin{proof}
Apply Lemma~\ref{lem:appendix-polygonal-approx}; its conclusions (iii)--(v) are the hypotheses of Lemma~\ref{lem:Dp-continuity}.
\end{proof}

\section*{Acknowledgments}
This article is based on the author's presentation at the Fifth Pan-Pacific International Conference on Topology and Applications (PPICTA 2026), held at Osaka Metropolitan University in September 2026.  The author thanks the organizers and participants for the stimulating discussions.  The author is especially grateful to Professor Mikami Hirasawa for a helpful comment that motivated the numerical comparison between the granny knot $3_1\#3_1$ and the square knot $3_1\#\overline{3_1}$ in the ropelength-filtered setting.

\section*{Declaration of generative AI and AI-assisted technologies in the manuscript preparation process}
During the preparation of this work, the author used ChatGPT (OpenAI) and Claude (Anthropic) for literature-search assistance, language refinement, structural editing, and assistance in drafting the numerical implementation reported in Subsection~\ref{subsec:granny-square}.  The numerical inputs and outputs, cited sources, and all mathematical definitions, statements, estimates, proofs, references, and final formulations were independently checked and revised by the author as needed.  The author takes full responsibility for the content of the publication.

\section*{Declaration of competing interest}
The author declares that he has no known competing financial interests or personal relationships that could have appeared to influence the work reported in this paper.

\section*{Data availability}
The numerical computations in Subsection~\ref{subsec:granny-square} use only the publicly available coordinate data cited as \cite{TightKnotData}.  No other data were used for the research described in this article.


\end{document}